\documentclass[onecolumn,draftcls,12pt]{IEEEtran}
\usepackage{amssymb,times}
\usepackage{cite, epsfig}
\usepackage{ verbatim, amsfonts,graphicx,enumerate}
\usepackage{bm,array}
\usepackage[cmex10]{amsmath}
\usepackage{bm}
\usepackage{framed}
\usepackage{graphicx}
\usepackage{enumitem}
\usepackage{epstopdf}
\usepackage{amsthm} 
\usepackage{thmtools, thm-restate}
\declaretheorem{theorem}
\usepackage{caption}
\usepackage{subcaption}
\usepackage{color} 
\usepackage{makecell}
\usepackage{enumitem}
\usepackage{algorithm}
\usepackage{varwidth}
\usepackage{algpseudocode}
\usepackage{tikz}

\theoremstyle{remark}

\newtheorem{lemma}{\bf Lemma}
\theoremstyle{plain}
\definecolor{mycolor1}{rgb}{0.97255,0.97255,0.97255}%
\newcommand{\distas}[1]{\mathbin{\overset{#1}{\kern\z@\sim}}}%

\def \P {{\mathbf P}}

\def \U {{\mathcal U}}

\def \x {{\mathbf x}}

\def \g {{\mathbf g}}
\def \y {{\mathbf y}}
\def \h {{\mathbf h}}

\def \p {{\mathbf p}}

\def \O {\mathcal{O}}

\def \E {{\mathbb E}}
\def \F {{\mathcal{F}}}
\def \L {{\boldsymbol{\Lambda}}}
\def \X {{\mathbf X}}
\def \Rn {{\mathbb R}}

\def \sk {{\sum\limits_{t=1}^{T}}}
\def \sk {{\sum\limits_{i=1}^{K}}}
\def \stau {{\frac{1}{T}\sum\limits_{t=1}^{T}}}

\def \ep {{\epsilon}}

\def \lams {{\boldsymbol{\lam}^{\star}}}

\def \EE {{\mathbb{E}}}

\def \lam {{\lambda}}
\def \lamb {{\boldsymbol{\lam}_{t}^{i-1}}}
\def \lambm {{\boldsymbol{\lam}_{t}^{i}}}
\def \lb {{\boldsymbol{\lam}}}

\DeclareMathOperator*{\argmin}{arg\,min}
\DeclareMathOperator*{\argmax}{arg\,max}

\def \Xc {{\mathcal{X}}}
\def \Pc {{\mathcal{P}}}

\def \lam {{\bm{\lambda}}}

\def \y {{\boldsymbol{\gamma}}}
\def \w {{\mathbf{w}}}
\def \x {{\mathbf{x}}}

\def \h {{\mathbf{h}}}

\def \Xc {{\mathcal{X}}}
\def \Pc {{\mathcal{P}}}

\def \lam {{\bm{\lambda}}}

\def \y {{\mathbf{y}}}
\def \u {{\mathbf{u}}}
\def \v {{\mathbf{v}}}
\def \w {{\mathbf{w}}}
\def \e {{\mathbf{e}}}

\def \bg {{\mathbf{h}}}

\def \Nn {{\mathbb{N}}}

\providecommand{\Ex}[1]{\mathbb{E}\left[#1\right]}

\providecommand{\px}[1]{P_{\L}\left[#1\right]}
\newcommand{\col}[1]{\textcolor{black}{#1}}

\newcommand{\colr}[1]{\textcolor{black}{#1}}

\providecommand{\abs}[1]{\left|#1\right|}
\providecommand{\norm}[1]{\left\|#1\right\|}
\providecommand{\ip}[1]{\langle#1\rangle}

\allowdisplaybreaks

\psfull
\begin{document}
	%
	\title{\vspace{-0.5cm}Asynchronous Incremental Stochastic  Dual Descent Algorithm for Network Resource Allocation}
\author{\IEEEauthorblockN{Amrit Singh Bedi, \emph{Student Member, IEEE} and Ketan Rajawat, \emph{Member, IEEE} \thanks{The	authors are with the Department of Electrical Engineering, IIT Kanpur, Kanpur
			(UP), India 208016 (email: {amritbd, ketan}@iitk.ac.in).}}\vspace{-1cm}}

	\maketitle
	
	\begin{abstract}
Stochastic network optimization problems entail finding resource allocation policies that are optimum on an average but must be designed in an online fashion. Such problems are ubiquitous in communication networks, where resources such as energy and bandwidth are divided among nodes to satisfy certain long-term objectives. This paper proposes an asynchronous incremental dual decent resource allocation algorithm that utilizes delayed stochastic {gradients} for carrying out its updates. The proposed algorithm is well-suited to heterogeneous networks as it allows the computationally-challenged or energy-starved nodes to, at times,  postpone the updates. The asymptotic analysis of the proposed algorithm is carried out, establishing dual convergence under both, constant and diminishing step sizes. It is also shown that with constant step size, the proposed resource allocation policy is asymptotically near-optimal. An application involving multi-cell coordinated beamforming is detailed, demonstrating the usefulness of the proposed algorithm. 
	\end{abstract}
	
	\begin{IEEEkeywords}
		Stochastic {subgradient}, resource allocation, asynchronous algorithm, incremental algorithm. \vspace{0mm} 
	\end{IEEEkeywords}

\section{Introduction}
The recent years have witnessed an unprecedented growth in the complexity and bandwidth requirements of network services. The resulting stress on the network infrastructure has motivated the network designers to move away from simpler or modular architectures and towards optimum ones. To make sure that resources such as bandwidth and energy are allocated efficiently, optimum designs advocate cooperation between the network nodes \cite{hong2007cooperative,sadek2007multinode}. This paper considers the problem of cooperative network resource allocation that arises in wireless communication networks \cite{georgiadis2006resource}, smart grid systems \cite{alloc_smart_tut_1}, and in the context of scheduling \cite{jaramillo2010optimal}. Of particular interest is the stochastic resource allocation problem, where the goal is to find an allocation policy that is asymptotically optimal \cite{stoch_res_prob_1,stoch_res_prob_3}. Although such problems are infinite dimensional in nature, they can be solved in an online fashion via stochastic dual descent methods, allowing real-time resource allocation that is also asymptotically near-optimal \cite{ribeiro2010separation,ale10,nedic2009approximate,yu2006dual,luo2008dynamic}. 

Heterogeneous networks are common to a number of applications where the energy availability, computational capability, or the mode of operation of the nodes is not the same across the network. Key requirements for heterogeneous network protocols include scalability, robustness, and tolerance to delays and packet losses. Towards this end, a number of distributed algorithms have been proposed in the literature \cite{veeravalli,nedic2001distributed,chang2015asynchronous,zhang2014asynchronous,wei20131,srivastava2011distributed, nedic2011asynchronous, ram2009asynchronous}. By eliminating the need for a fusion center, the distributed algorithms operate with reduced communication overhead, and render the network resilient to single-point failures. 

Most distributed algorithms still place stringent communication and computational requirements on the network nodes. For instance, the dual stochastic {gradient} methods entail multiple updates and message exchanges per time slot, and cannot handle missed or delayed updates. In heterogeneous networks, such delays are often unavoidable, arising due to poor channel conditions, traffic congestion, or limited processing power at certain nodes. This paper proposes a distributed asynchronous stochastic resource allocation algorithm that tolerates such delays. The next subsection outlines the main contributions of this paper.\vspace{0cm} 
\vspace{-3mm}
\subsection{Contributions and organization}
The stochastic resource allocation problem is formulated as a constrained optimization problem where the goal is to maximize a network-wide utility function. The allocated resources at the different nodes in the network are coupled through constraint functions that involve expectations with respect to a random network state. Specifically, the aim is to find an allocation policy that satisfies the constraints on an average. The distribution of the state variables is not known, so that the optimization problem does not admit an offline solution. Instead, the idea is to observe the instances of the state variables over time, and allocate resources in an online manner. It is well-known that stochastic dual descent algorithms yield viable online algorithms for such problems \cite{ale10}. 

Within the heterogeneous network setting considered here, the focus is on distributed algorithms that can tolerate communication and processing delays {\cite{rajawat2011cross,chang2015asynchronous,agarwal2011distributed}}. Different from the state-of-the-art algorithms that utilize the standard stochastic { gradient} methods {\cite{ribeiro2010separation,agarwal2011distributed,gatsis2010class}}, we develop two variants of the asynchronous dual descent algorithm that allow some of the nodes in the network to temporarily ``fall back,'' in the event of low energy availability, unusually large processing delay, node shutdown, or channel impairments. The first asynchronous variant utilizes a fusion center to collect the possibly delayed { gradient}s from various nodes and carry out the updates {(cf. \ref{sec-asy} )}. The second variant eliminates the need for the fusion center, and instead utilizes the fully distributed and incremental stochastic { gradient} descent algorithm, where the nodes carry out updates in a round-robin fashion and pass messages along a cycle {(cf. \ref{sec-inc} )}. As earlier, the use of stale { gradient}s for primal and dual updates, allows the algorithm to be run on two different clocks, one corresponding to the local resource allocation and tuned to the changing random network state, while the other dictated by the message passing protocol. The key feature of the proposed algorithm is the possibility for the second clock to slow down temporarily and wait for slower nodes to catch up. The proposed algorithm thus allows timely resource allocation, while tolerating occasional delays in message passing.

The asymptotic performance of the proposed algorithm is studied under certain regularity conditions on the problem structure and bounded delays. In particular, the asymptotic performance of the asynchronous incremental stochastic { subgradient} descent {(AIS-SD)} algorithm is characterized under both, diminishing and constant step-sizes. The overall structure of the proof is based on the convergence results in the incremental stochastic { subgradient} descent algorithm of \cite{veeravalli}  and the asynchronous incremental {{subgradient}} method of \cite{nedic2001distributed}. Specific to the resource allocation problem at hand, the asymptotic near-optimality and almost sure feasibility of the primal allocation policy is established for the case of constant step sizes. {When applied to resource allocation problems, the proposed algorithm is called asynchronous incremental stochastic dual descent (AIS-DD)}. It is remarked that since the proposed algorithms utilize stochastic { subgradient} descent, their computational complexity is also comparable to other distributed stochastic algorithms \cite{veeravalli,agarwal2011distributed,duchi2015asynchronous,chang2015asynchronous,zhang2014asynchronous,wei20131,srivastava2011distributed, nedic2011asynchronous, ram2009asynchronous,sirb2016decentralized}. The calculation of the { subgradient} is the most computationally expensive step, and like other first-order algorithms, must be carried out at every time slot.

Finally, the stochastic coordinated multi-cell beamforming problem is formulated and solved via the proposed algorithm. Detailed simulations are carried out to demonstrate the usefulness of the proposed algorithm in delay-prone and distributed environments. Summarizing, the main contributions of the paper include (a) the {AIS-SD} algorithm and its convergence (b) primal near-optimality and feasibility results for the allocated resources using {AIS-DD}; and (c) demonstration of the proposed algorithm on a practical stochastic coordinated multi-cell beamforming problem. 
	 
	The rest of the paper is organized as follows. Sec. \ref{rel} provides an outline of the related literature. Sec. \ref{probfor} describes the problem formulation and recapitulates the known results. Sec. \ref{prop} details  the proposed algorithm. Sec. \ref{sec-conv} lists the required assumptions, and provides the primal and dual convergence results. Sec. \ref{beam} formulates the stochastic version of the coordinated beamforming problem along with the relevant simulation results. Finally, Sec. \ref{conc} concludes the paper. \vspace{-3mm}
	
\subsection{Related work} \label{rel}
Resource allocation problems have been well-studied in the context of cross-layer optimization in networks \cite{lin2006tutorial}. Popular tools for solving stochastic resource allocation problems include the backpressure algorithm \cite{georgiadis2006resource} and variants of the stochastic dual descent method \cite{gatsis2010class,ale10}. However, most of these works only consider synchronous algorithms, and the effect of communication delays has not been examined in detail. An exception is the asynchronous {subgradient} method proposed in \cite{rajawat2011cross}, where delayed {subgradient}s were utilized for resource allocation. The present work extends the algorithm in \cite{rajawat2011cross} by allowing delayed stochastic {subgradient}s. Additionally, the proposed algorithm is also incremental, and is therefore applicable to a wider variety of problems. 

Depending on the mode of communication among the nodes, distributed algorithms can be broadly classified into three categories, namely, diffusion, consensus, and incremental \cite{lopes2007incremental}. Of these, the incremental update rule generally incurs the least amount of message passing overhead \cite{rabbat2005quantized}, and is of interest in the present context. The incremental {subgradient} descent and its variants have been widely applied to large-scale problems, and generally exhibit faster convergence than the traditional steepest descent algorithm and its variants \cite{nedic2001incremental}.

The stochastic {gradient} and {subgradient} algorithms are well-known within the machine learning and signal processing communities \cite{bottou2010large,bach2011non,agarwal2011distributed}. The incremental stochastic {subgradient} method, with cyclic, random, and Markov incremental variants, was first proposed in \cite{veeravalli}. The asymptotic analysis of dual problem in the present work follows the same general outline as that of the cyclic incremental algorithm in \cite{veeravalli}, with additional modifications introduced to handle asynchrony. It is emphasized that these modifications are not straightforward, since the delayed stochastic {subgradient} is not generally a descent direction on an average. The present work also allows delays in both, primal and dual update steps, and establishes asymptotic near-optimality and feasibility of the primal allocation policies. Finally, saddle point algorithms have recently been applied to unconstrained \cite{rev_ref2} or proximity-constrained \cite{rev_ref3} network optimization problems, but do not readily generalize to the general form constrained optimization problem considered here.

Asynchronous algorithms have also been considered within the Markov decision process framework \cite{neely2012asynchronous}, though the setup there is quite different and does not apply to the problem at hand. On the other hand, asynchronous first order methods have attracted a significant interest from the machine learning community {\cite{agarwal2011distributed,duchi2015asynchronous}}. For problems where the exact {subgradient} is available at each node, the asynchronous alternating directions method of multiplier (ADMM) has been well-studied \cite{chang2015asynchronous,zhang2014asynchronous,wei20131}. The present work considers stochastic algorithms, and thus differs considerably in terms of both analysis and the final results. Even among algorithms utilizing stochastic {subgradient}s, the definition of asynchrony varies across different works. One way to model asynchrony is to allow each node to carry out its update according to a local Poisson clock. This approach is followed in \cite{srivastava2011distributed, nedic2011asynchronous, ram2009asynchronous}, all of which consider various consensus-based distributed {subgradient} algorithms. The asynchronous adaptive algorithms in \cite{sayed2015asynchronous} also subscribe to the same philosophy, with decoupled node-updates due to communication errors, changing topology, and node failures. The incremental algorithm considered here is very different in terms of operation and analysis.

On the other hand, asynchronous operation can be modeled via delayed {gradient}s or {subgradient}s utilized for the updates. A  consensus-based stochastic algorithm proposed in \cite{sirb2016decentralized}, and utilizes randomly delayed stochastic {gradient}s. Along similar lines, asynchronous saddle point algorithms for network problems with edge-based constraints have recently been proposed \cite{bedi2017beyond1,bedi2017beyond2}. \colr{Finally, for the unconstrained variants of the problem, a non-parametric approach has been proposed in \cite{koppel2016parsimonious}.} Different from these works, the network resource allocation framework considered here allows generic convex constraints. Further, the incremental algorithm developed here handles stale {subgradient}s while incurring significantly lower communication overheads. Asynchronous variants of the classical or averaged stochastic {gradient} methods have been proposed in \cite{adadelay, agarwal2011distributed, feyzmahdavian2015asynchronous,liu2015asynchronous}. The generic problem of interest here is that of the minimization of a sum of private functions at various nodes. Further, a network with star topology is considered, with updates being carried out using delayed {{ gradient}}s collected at the fusion center. Different from these works, the proposed algorithm is incremental, does not require a fusion center, and is therefore more relevant to the network resource allocation problem at hand. Unlike these works, the present work also avoids making any assumptions on the compactness of the domain of the dual optimization problem. Before concluding, it is remarked that this work develops convergence results that hold on an average. Stronger results, where convergence is established in an almost sure sense, require a more involved analysis, and are not pursued here. 


The notation used in this paper is as follows. Scalars are represented by small letters, vectors by small boldface letters, and constants by capital letters. The index $t$ is used for the time or iteration index. The inner product between vectors $\boldsymbol{a}$ and $\boldsymbol{b}$ is denoted by $\ip{\boldsymbol{a},\boldsymbol{b}}$. For a vector $\x$, projection onto the non-negative orthant is denoted by $[\x]^+$. The expectation operation is denoted by $\EE$. \colr{The notation $\nabla$ is used for gradient and $\partial$ is used for the subgradient.} \colr{The Euclidean norm is denoted by $\norm{\cdot}$.}
\vspace{-5mm}		
\section{Problem formulation}\label{probfor}

\subsection{Problem statement}\label{key}
This section details the stochastic resource allocation problem at hand for a network with $K$ nodes. The stochastic component of the problem is captured through the random network state, comprising of the random vectors $\bg^i \in \Rn^{q}$ for each node $i \in \{1, \ldots, K\}$, with unknown distributions. The overall problem is formulated as follows. \vspace{0mm}
	\begin{subequations}\label{P1}
		\begin{align}\label{p1}
	\textsf{P}\ := \ \max &\ \  \sk f^i(\x^i) \\
	\text{s.t. } \ \  &\sk \u^i(\x^i) + \EE\left[\v^i(\bg^i,\p^i_{\bg^i})\right]\succeq 0\label{const_1}\\
	& \x^i\in\Xc^i, \ \p^i\in\Pc^i\label{const_2}
	\end{align}
	\end{subequations}
\colr{where the expectation in \eqref{const_1} is with respect to the random vector $\h^i$ and $\mathsf{P}$ is finite.} The optimization variables in \eqref{P1} include the resource allocation variables $\{\x^i\!\!\in\!\!\mathbb{R}^n\}_{i=1}^K$ and the policy functions $\{\p^i\!\!:\!\!\Rn^{q} \rightarrow \Rn^{p}\}_{i=1}^K$, under the constraints \eqref{const_1}-\eqref{const_2}.  Note that, the constraints in \eqref{const_1} are required to be satisfied on an average, whereas those in \eqref{const_2} are needed to be satisfied instantaneously. The functions $f^i:\Rn^n \rightarrow \Rn$ are assumed to be concave, and the sets $\Xc^i \subseteq \Rn^n$, convex and compact. The constraint function at node $i$ is vector-valued, and is given by $\u^i(\x^i):=[u^i_1(\x^i) \cdots u^i_d(\x^i)]^T$, where $\{u_k^i(\x^i):\Rn^n \rightarrow \Rn\}_{k=1}^d$ are concave functions. On the other hand, no such restriction is imposed upon the vector-valued function $\v^i:\Rn^{p} \times \Rn^{q} \rightarrow \Rn^d$ and the {compact} set of functions $\{\mathcal{P}^i\}_{i=1}^K$. Of course, the overall problem still needs to adhere to certain regularity conditions (see Sec. \ref{sec-conv}), such as the Slater's constraint qualification and Lipschitz continuity of the {{ gradient}} function; see (\textbf{A1})-(\textbf{A7}). 
	
Since the distribution of $\bg^i$ is also not known in advance, it is generally not possible to solve for $\textsf{P}$ in an offline manner. Therefore, an online algorithm is sought to solve problem `on the fly' as the independent identically distributed (i.i.d.) random variables $\{\bg_t^i\}_{t\in \Nn}$ are realized and observed. For brevity, we denote $\p_t^i:= \p_{\bg^i_t}^i$ and $\g_t^i(\p_t^i,\x^i):=\u^i(\x^i) + \v^i(\bg_t^i,\p^i_{\bg^i_t})$. Therefore, it is possible to write \eqref{const_1} equivalently as $\EE\left[\sk \g_t^i(\p_t^i,\x^i)\right]\succeq 0$. The algorithm outputs a sequence of vector pairs $\{\x_t^i, {\p}_t^i\}_{t}$, that are used for allocating resources in a timely manner. Towards this end, the stochastic dual descent algorithm has been proposed in \cite{ale10}, which yields allocations that are almost surely near-optimal and provably convergent.

	
In the present paper, the focus is on networked systems where both, allocations {$(\x^i,\p^i)$} and the functions $f^i$ and $\g^i_t$ are private to each node $i$. Likewise, the random variable $\bg_t^i$ is also observed and estimated locally at each node $i$. In other words, while the nodes can exchange dual variables and numerical values of the {{ gradient}}s, they may not be willing to reveal the full functional form of the objective or constraint functions and other locally estimated quantities, owing to privacy and security concerns. Such \emph{privacy-preserving} cooperation is common for many secure multi-agent systems \cite{wei20131,wei2012distributed,mota2013d}. To this end, the nodes may be arranged in a star topology, and utilize a centralized controller for collecting and distributing various algorithm iterates. Alternatively, ring topology may be used, allowing a fully distributed implementation, where the exchanges occur only between two immediate neighbors.
 In order to clarify the problem formulation considered in \eqref{P1}, the following simple example is considered. \\
 \textbf{Example 1.} Consider the problem of network utility maximization over a wireless network consisting of $K$ nodes. The aim is to maximize the network-wide utility given by
 \begin{align}\label{util}
 \sum\limits_{i=1}^{K}U(r^i)
 \end{align} 
 where $U(\cdot)$ is a concave function that quantifies the utility obtained by the node $i$ upon achieving a rate $r^i \in[r_{\min},r_{\max}]$. The channel is assumed to be time-varying, and for each channel realization $h^i$, node $i$ allocates the power $p^i_{h^i}$, achieving the instantaneous rate of $\log(1+h^ip^i_{h^i})$, where the noise power is assumed to be one. The goal is to maximize the utility in \eqref{util} subject to constraints on the average rate and the average power consumption, and the full problem can be written as (cf. \eqref{P1}):
 \vspace{-5mm}
 \begin{subequations}\label{example2}
 	\begin{align}
 	\max\limits_{r^i,p^i}   \ \ \ 	\sum\limits_{i=1}^{K}U(r^i)\hspace{2cm}&\\
 	\text{s.t. }	\Ex{\sum\limits_{i=1}^{K}\left(\frac{1}{2}\log(1+h^ip^i_{h^i})\right)}&\geq \sum\limits_{i=1}^{K}r^i\label{rate_costrainrs2}\\
 	\Ex{\sum\limits_{i=1}^{K}p^i_{h^i}}&\leq P_{\max}\label{power_const2}\\
 	{r^i\in[r_{\min},r_{\max}]}, \ p^i\in\mathcal{P}^i
 	\end{align}
 \end{subequations}
 It is remarked that $\mathcal{P}^i$ is a set of functions $p^i:\Rn \rightarrow \Rn$, while $p^i_{h^i}$ is a random variable that depends on $h^i$. That is, the optimization variables in \eqref{example2} include the rates $r^i$ and the power allocation functions $p^i$.	
\vspace{0mm}	
\subsection{Existing approaches and challenges}\label{challenges}
 We begin with explicating the desirable features of an algorithm that seeks to solve \eqref{P1}. Specifically, it is required that any such algorithm meets the following requirements.
 \begin{enumerate}
 	\item[\textbf{F1}.] The algorithm should allow nodes to ``fall behind'' temporarily, e.g., under poor channel conditions and intermittent transmission failures.
 	\item[\textbf{F2}.] The algorithm should allow a distributed implementation, that is, without requiring a star-topology or an FC.
 \end{enumerate}\vspace{0mm}
 These features are particularly important for large and heterogeneous networks where delays may be unavoidable and designating an FC may be impractical. Put differently, (\textbf{F1}) requires the algorithm to handle the inevitable delays that may occur due to temporarily poor channel conditions or noise. Complementarily, (\textbf{F2}) is an  architectural requirement that must be kept in mind when choosing or designing the algorithm.

	Since the number of constraints in \eqref{const_1} are finite, the problem is more tractable in the dual domain. To this end, introducing a dual variable $\lam\in\mathbb{R}^d_+$ corresponding to constraint in \eqref{const_1}, the stochastic (sub-)gradient descent method was proposed for solving such problems in \cite{ale10}. The Lagrangian of {\eqref{P1} is given by}
	\begin{align}\label{lagrangian}
	L(\lb,\X,\P)& =\sk \left\{f^i(\x^i)+\ip{\lb,\EE\left[ \g_t^i(\p_t^i,\x^i)\right]}\right\}
	\end{align}
	where $\X$ and $\P$ collect the primal optimization variables $\{\x^i\}_{i=1}^{K}$ and $\{\p^i\}_{i=1}^K$ respectively. Next, the dual function is obtained by maximizing $L$ with respect to $\X$ and $\P$. Since the Lagrangian is expressed as a sum of $K$ terms, each depending on a different set of variables, the maximization operation is separable and the dual function takes the following form:
	\begin{align}\label{dual_exp_1}
	D(\lb) &= \sum_{i=1}^K \max_{\x^i\in\Xc^i,\ \p^i\in\Pc^i} \left[f^i(\x^i)+\ip{\lb,\left(\EE\left[ \g_t^i(\p_t^i,\x^i)\right]\right)}\right] \nonumber\\
	&=: \sum_{i=1}^K D^i(\lb).\normalsize
	\end{align}
	The dual problem is given by
		\vspace{0mm}
	\begin{align}\label{dual_problem}
	\textsf{D} &= \min_{\lb\in\mathbb{R}^d_+} \sk D^i(\lb).
	\end{align}
While for general problems, it only holds that $\mathsf{D} \geq \mathsf{P}$, that the stochastic resource allocation problem considered here has a zero duality gap, i.e., $\mathsf{P} = \mathsf{D}$ \cite[Thm. 1]{ribeiro2010separation}. The result \colr{utilizes the Lyapunov's convexity theorem and holds under strict feasibility (Slater's condition), bounded {subgradient}s, and continuous cumulative distribution function of $\h^i$ for each $i$.} It is remarked that similar results are well-known in economics \cite{shitovitz1973oligopoly}, wireless communications \cite{luo2008dynamic,ribeiro2010separation}, and control theory \cite{hermes1969functional}.

	The result on zero duality gap legitimizes the dual descent approach, since the dual problem is always convex, and the resultant dual solution can be used for primal recovery. To this end, similar problems in various contexts have been solved via the classical dual descent algorithm \cite{ale10,OFDM_gian, rajawat2011cross, ribeiro2010separation, marques3}, wherein the primal updates utilize various sampling techniques. \colr{It is remarked however that from a practical perspective, solving the dual problem alone is not sufficient, since online allocation of power or rate variables necessitates determining the primal optimum variables $\{{\x^i}^\star, {\p^i}^\star\}$. In the present case, since $\p^i$ is infinite dimensional, primal recovery and consequently, online resource allocation, is not straightforward. }  
	

%
Since the distribution of $\h^i$ is not known in advance, solving \eqref{dual_problem} via classical first or second order descent methods requires a costly Monte Carlo sampling step \cite{gatsis2010class}. Instead, the use of stochastic {subgradient} descent has been proposed in \cite{ale10,wang2011resource}, which takes the following form for $t\geq 1$,
	\begin{enumerate}
		\item[\textbf{D1}.] \textbf{Primal updates:} At time $t$, node $i$ observes or estimates ${\bg_t^i}$, and allocates the resources in accordance with:
		\begin{align}\label{primal_1}
		\hspace{-1cm}    \{{\x^i_t(\lb_{t})}, \p^i_t(\lb_{t})\} &= \argmax_{\x \in \Xc^i, \mathring{\p} \in \Pi^i_t}  f^i(\x)+\ip{\lb_{t}, \g_t^i(\mathring{\p},\x)}
		\end{align}
		\item[\textbf{D2}.] \textbf{Dual update:} The dual updates at time $t$ take the form:
		\begin{align}\label{dual}
		\lb_{t+1}=\Big[\lb_{t}-\ep\sk \g_t^i(\p_t^i(\lb_{t}),{\x^i_t (\lb_{t})})\Big]^+.
		\end{align}
	\end{enumerate}
	Here, $\Pi_t^i :=\{\p^i_{\bg^i_t} \in \Rn^{p} | \p^i \in \Pc^i\}$ is the set of all legitimate values of the vector $\p^i_{\bg^i_t}$. The term $\g_t^i(\p_t^i(\lb_{t}),\x^i_t (\lb_{t}))$ is a stochastic {{gradient}} of the dual function $D^i(\lam)$ at $\lb=\lb_t$. Further for notational brevity, $\g_t^i(\lb):=\g_t^i(\p_t^i(\lb),\x^i_t (\lb))$ is used throughout the paper. Recall that for a given $\lb$, $\g_t^i(\lb)$ is stochastic and depends on the random variable $\h^i_t$, as discussed in Sec. \ref{key}. The algorithm is initialized with an arbitrary $\lb_1$ and the resulting allocations are asymptotically near optimal and feasible. A constant step-size stochastic {gradient} descent algorithm is utilized in the dual domain, which not only allows recovery of optimal primal variables via averaging, but also bestows it the ability to handle small changes in the network topology or other problem parameters. The algorithm can be implemented in a distributed fashion in a network with star-topology, with the help of a fusion center (FC). Within the FC-based implementation, the primal iterates are calculated and used locally at each node $i$. At the end of each time slot, the node $i$ communicates the {gradient} component ${\g^i_t(\lam_t)}$ to the FC, which carries out the dual update \eqref{dual} and broadcasts $\lb_{t+1}$ to all the nodes in the network. Summarizing, the stochastic algorithm is preferred over its deterministic counterpart since it does not require Monte Carlo iterations, yields asymptotically near-optimal resource allocations, and is provably convergent if the stochastic process $\{\bg^i_t\}$ is stationary. 
	
\colr{It is remarked that since \eqref{P1} is infinite dimensional, full primal recovery is generally not possible using such dual methods. Existing algorithms only allow partial primal recovery, as will also be possible via Theorem 2. Specifically, it is well-known that while the running average $\frac{1}{T}\sum_{t = 1}^T \x^i_t$ can be viewed as the approximate version of the primal optimum ${\x^i}^\star$, no such interpretation exists for the infinite-dimensional variable $\p^i$.} \colr{For instance, the running average of $\p^i_t$ cannot be meaningfully related to the corresponding optimum ${\p^i}^\star$ \cite{ale10,wang2011resource}. Nevertheless, the resource allocation carried out using the primal iterates $\{\x^i_t(\lb_t),\p^i_t(\lb_t)\}$ still ensures near-optimality and asymptotic feasibility (cf. Theorem 2). }	

In view of the desiderata (\textbf{F1})-(\textbf{F2}), observe that a network implementation of \eqref{primal_1}-\eqref{dual} is still impractical since it is synchronous and FC-based, and thus has relatively stringent communication requirements. In particular, the algorithm necessitates that each node exchanges messages ({i.e. $\g^i_t(\lam_t)$ \& $\lb_t$}) with the FC at every time-slot, thereby incurring a large communications cost. Since the updates \eqref{primal_1}-\eqref{dual} must occur before the network state changes, the nodes must synchronize and cooperate in order to meet these deadline constraints, ultimately increasing message passing overhead and consuming more energy. Further, nodes in large networks are often heterogeneous, and may not always be able to transmit the {gradients} within the stipulated time. Finally, if the nodes are not deployed in a star-topology around the FC, the need for multi-hop communications further increases the delays, results in heterogeneous energy consumption, and increases protocol overhead. In all such cases, the FC must wait for the updates to arrive from all the nodes, possibly requiring all the nodes to skip resource allocation for one or more time slots, and resulting in a suboptimal asymptotic objective value.
	
	\section{Proposed Algorithm}\label{prop}
	This section details the proposed stochastic dual descent algorithm that incorporates the features (\textbf{F1})-(\textbf{F2}) in its design. To begin with, Sec. \ref{sec-asy} describes the asynchronous variant that tolerates delayed {gradients} still resulting in near-optimal resource allocation. Next, Sec. \ref{sec-inc} details the more general AIS-DD algorithm that is amenable to a distributed implementation. 
\vspace{-3mm}	
	\subsection{Asynchronous stochastic dual descent}\label{sec-asy}
	The asynchronous stochastic dual descent algorithm addresses (\textbf{F1}), and proceeds as follows for all $t \geq 1$:
	\begin{enumerate}
		\item  \textbf{Primal update}: At each time $t$, node $i$ solves
		\begin{align}\label{primal1_1}
		\{\x_t^i(\lb_{t-\pi_i(t)})&, \p_t^i(\lb_{t-\pi_i(t)})\}\nonumber\\
		&\hspace{-1cm}:=\argmax_{\col{\x\in\boldsymbol{\mathcal{X}^i},\mathring{\p}\in\Pi_t^i}}  f^i(\x)+\ip{\lb_{t-\pi_i(t)},\col{\g_t^i(\mathring{\p},\x)}}
		\end{align}
		for all $1 \leq i \leq K$, and some \col{finite} delay $\pi_i(t) \geq 0$. 
		\item  \textbf{Dual update}: The dual update at time $t$ is given by
		\begin{align}\label{dual_1}
		\lb_{t+1}=\left[\lb_t-{\ep}\left(\sk \g_{t-\delta_i(t)}^i(\lb_{t-\tau_i(t)})\right)\right]^+
		\end{align}
		\normalsize
	\end{enumerate}
the stale {{ gradient}}, evaluated at time $t-\delta_i(t)$, is given by  \begin{align}
\g_{t-\delta_i(t)}^i&(\lb_{t-\tau_i(t)})\nonumber
\\
&\!\!\!\!\!\!\!\!:=\g_{t-\delta_i(t)}^i(\p_{t-\delta_i(t)}^i(\lb_{t-\tau_i(t)}),\x_{t-\delta_i(t)}^i(\lb_{t-\tau_i(t)}))\label{for_delay}
\end{align} where the total delay is denoted by  $\tau_i(t):=\pi_i(t)+\delta_i(t)$ and $\pi_i(t),~\delta_i(t) \geq 0$. 
%
			
Different from \eqref{primal_1}, the resource allocation in \eqref{primal1_1} utilizes an old dual variable, $\lb_{t-\pi_i(t)}$.	Further, the dual update is also carried out using an old {gradient} {$\g_{t-\delta_i(t)}^i(\lb_{t-\tau_i(t)})$}. The two modes of asynchrony introduced in \eqref{primal1_1}-\eqref{dual_1} allow the primal and dual updates to be carried out at different time scales. In other words, while the resource allocation at each node still occurs at every time slot, the rate at which the dual variables and the {{ gradient}}s are exchanged may be different. In order to highlight the asynchronous nature of the algorithm, the implementation of \eqref{primal1_1}-\eqref{dual_1} is now described from the perspective of the FC and that of node $i$, in Algorithms \ref{algo_1_a} and \ref{algo_1_b}, respectively.
		\vspace{0mm}
		\col{\begin{algorithm}
				\caption{: \textbf{Operation at FC}}\label{algo_1_a}
				\col{	\begin{itemize}
						\item[(S0)] \textbf{Initialize:} $t=1$, $\lam_1$, ${\ep}$.
						\item[(S1)] \textbf{Update} the dual variable $\lam_t$ as in \eqref{dual_1} using the latest available {{ gradient}}s $\g_{t-\delta_i(t)}^i(\lb_{t-\tau_i(t)})$ for each $1\!\leq\! i \leq K$.
						\item[(S2)] (Optional) \textbf{Broadcast} the updated $\lb_t$ to all the nodes. 
						\item[(S3)] ({{Optional}}) \textbf{Listen} for updated { gradient}s from all the nodes until a time-out. 
						\item[(S4)] $t=t+1$, go to (S1).
					\end{itemize}}
				\end{algorithm}\vspace{-4mm}}
				\col{\begin{algorithm}
						\caption{: \textbf{Operation at node $i$}}\label{algo_1_b}
						\col{	\begin{itemize}
								\item[(S0)] \textbf{Initialize:} $t=1$.
								\item[(S1)] \textbf{Estimate} the associated  random parameter $\h_t^i$. 
								\item[(S2)] \textbf{Allocate} resources using the latest available $\lb_{t-\pi_i(t)}$ as in  \eqref{primal1_1}.
								\item[(S3)] ({{Optional}}) \textbf{Transmit} the { gradient} \col{$\g_t^i(\lb_{t-\pi_i(t)})$} to the FC.
								\item[(S4)] ({{Optional}}) \textbf{Listen} for $\lb_t$ during the rest of the time slot. Only the latest copy of $\lb_t$ is retained in the memory.
								\item[(S5)] $t=t+1$, go to (S1).
							\end{itemize}}
						\end{algorithm}\vspace{0mm}}
						
Observe that in Algorithms 1 and 2, some steps are `optional,' which in the present case, means that they can, at times, be skipped. These steps are however still required to be carried out `often enough', so that the total delay $\tau_i(t)$ is bounded for each node $i$; cf. (\textbf{A4}) in Sec. \ref{known}. Nevertheless, the optional steps in these algorithms allow the dual updates to occur at a different rate. For instance, as long as each packet is correctly time-stamped, the dual updates at the FC may occur as and when the { gradient}s become available, instead of following a fixed schedule.

						The ability to postpone or skip transmissions is important in the context of large heterogeneous networks. For instance, transmissions from the nodes to the FC often requires a multiple access protocol, inter-node coordination, and energy budgeting at each node. Consequently, energy-constrained nodes may extend their lifetime simply by scheduling their transmissions once every few time slots. Similarly, energy harvesting nodes may only transmit when sufficient energy is available, choosing to stay silent in times of energy paucity. The slower nodes may even skip the { gradient} calculation, as long as the resources are allocated in time. 						
						Finally, the communication between the nodes and the FC may also incur delays, arising from queueing, processing, or retransmission at various layers in the protocol stack. The flexibility of carrying out updates with stale information makes the network tolerant to such delays. 
						
						\begin{figure}	
						\centering
							\includegraphics[scale=0.33]{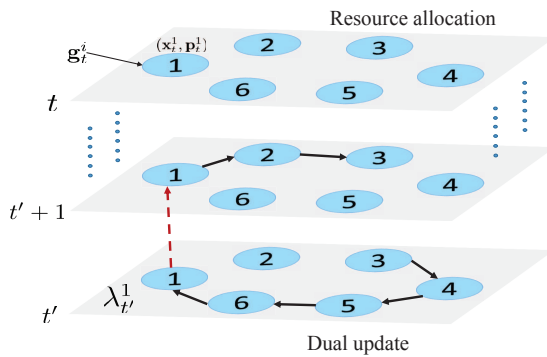}
							\caption{\colr{AIS-DD} algorithm operation.}
							\label{Diagram}	
						\end{figure}
						\vspace{0mm}
						\subsection{Asynchronous Incremental Stochastic Dual Descent}\label{sec-inc} 
						This subsection details an incremental version of the asynchronous algorithm introduced in Sec. \ref{sec-asy}, that obviates the need for an FC and is thus endowed with both (\textbf{F1}) and (\textbf{F2}). The \colr{AIS-DD} algorithm allows each node to perform the partial dual update itself, while passing messages to nodes along a cycle. Specifically, for a network with a ring topology, such that node $i$ passes dual variable $\lb_t^i$ to node $i+1$ and so on, the primal and dual updates take the following form. 
						
						\begin{enumerate}
							\item \textbf{Primal update: }  At time $t$, node $i$ solves
							\begin{align}\label{primal_sub}
							\hspace{-1cm}(\x^i_t&{(\lb^{i-1}_{t-\pi_i(t)})}, \p_{t}^i{(\lb^{i-1}_{t-\pi_i(t)})})\nonumber\\&:=\argmax_{{\x \in\mathcal{X}^i,\  \col{\mathring{\p}\in\Pi_t^i}}} f^i({\x})+\ip{\lb^{i-1}_{t-\pi_i(t)},\g_t^i(\col{\mathring{\p},\x})}.
							\end{align}
							\item \textbf{Dual update: }	At time $t$, the dual update at node $i$ takes the form
							\begin{align}\label{dual_sub}
							\lb^{i}_t=& \left[\lb^{i-1}_t-\ep  [\g_{t-\delta_i(t)}^i(\lb^{i-1}_{t-\tau_i(t)})]\right]^+
							\end{align} 
						\end{enumerate}
						 where,\vspace{0mm} \begin{align}\label{definition_gradient}
						 \g_{t-\delta_i(t)}^i&(\lb^{i-1}_{t-\tau_i(t)})\nonumber
						 \\
						 :=&\g_{t-\delta_i(t)}^i(\p_{t-\delta_i(t)}^i(\lb^{i-1}_{t-\tau_i(t)}),\x_{t-\delta_i(t)}^i(\lb^{i-1}_{t-\tau_i(t)}))
						 \end{align} and $\lb_t^0$ is read as $\lb_{t-1}^K$ and $\lb_t=\lb_t^0$ will be used to evaluate the performance of the asynchronous incremental  algorithms.
						A key feature of the \colr{AIS-DD} algorithm is that the message passing and the dual updates occur in parallel with the resource allocation, as shown in Fig.~\ref{Diagram}. The full implementation details are provided in Algorithm \ref{algo_3}. 
						\begin{algorithm}
							\caption{: \textbf{At node $i$}} \label{algo_3}
								\begin{itemize}
								\vspace{0mm}
								\item[(S0)] \textbf{Initialize:} $t=1, \lam_1^{i-1}$, $\ep$.
								\item[(S1)] \textbf{Estimate} the associated  random parameter $\h_t^i$.
								\item[(S2)] \textbf{Allocate} resources  using the latest available $\lb_{t-\pi_i(t)}^{i-1}$ as in \eqref{primal_sub}. 
								\item[(S3)] ({{Optional}}) \textbf{Receive} $\lb_{t'}^{i-1}$ and carries out the update  \eqref{dual_sub} using an older { gradient} $\g_{t'-\delta_i(t')}^i(\lb^{i-1}_{t'-\tau_i(t')})$.
								\item[(S4)] ({{Optional}}) \textbf{Transmit} an updated $\lb_{t'}^i$ to node $i+1$. 
								\item[(S5)] $t=t+1$, go to (S1).
							\end{itemize}
						\end{algorithm}
								
								Here, the two optional steps may be repeated as long as the received {$\lb_{t'}^{i-1}$} is still old, that is, $t' \leq t$. As in Sec. \ref{sec-asy}, the nodes are allowed to halt the updates temporarily, as long as they ``catch up,'' eventually. In other words, the updates for time $t'$ must be carried out before time $t'+\tau$ so as to ensure that $\tau_i(t) \leq \tau$ for all $t$. Interestingly, although resources are allocated at every time slot, the network may or may not carry out one or more message passing rounds per time-slot. It is remarked that the update in \eqref{dual_sub} must still be performed once at every node for each time index $t'$. Equivalently, the algorithm runs on two `clocks,' one dictating the resource allocation and synchronous with the changes in the network state, and the other governed by the rate at which messages get passed around the network. In the next section, we will establish that the such an algorithm still converges, as long as the difference between the two clocks is bounded. In summary, the \colr{AIS-DD} algorithm has all the benefits of the asynchronous dual descent algorithm of \eqref{primal1_1}-\eqref{dual_1}, while allowing a distributed implementation. 
								
As with classical incremental algorithms, the nodes must communicate along a ring topology. Strictly speaking, the message passing overhead is minimized if the updates occur along a {Hamilton cycle \cite{rabbat2005quantized}}. Even when the network does not admit a Hamilton cycle, an approximate cycle can be found using a random walk protocol \cite{eshragh2011hamiltonian} or the protocol described in {\cite[Sec. VII]{rabbat2005quantized}}. It is remarked that such a route need only be found once, at the start of the algorithm.
							\vspace{0mm}
								\section{Convergence results}\label{sec-conv}
								This section provides the convergence results for the {AIS-SD and AIS-DD} algorithm. We begin with developing and analyzing the convergence of {AIS-SD algorithm}  (cf. Theorem \ref{dual_opt}). It is emphasized that the {AIS-SD} algorithm is general-purpose, and can be used to minimize any sum of functions in an incremental and asynchronous manner.  Subsequently, the asynchronous incremental stochastic gradient descent algorithm is applied to \eqref{P1} in the dual domain, and a primal-averaging method is proposed that yields asymptotically near-optimal allocations (cf. Theorem \ref{primal_opt}). We begin with stating the assumptions and briefly reviewing some of the known results (Sec. \ref{known}). The results for the dual case are outlined in Sec. \ref{sec-dual}, while the near-optimality of the resource allocation is established in Sec. \ref{sec-pri}. 
								\vspace{0mm}
								\subsection{Assumptions and known results}\label{known}
								This subsection begins with the discussion of the following general optimization problem:\vspace{0mm}
								\begin{align}\label{dual_problem2}
								\textsf{D} &= \min_{\lb\in\L} \sk D^i(\lb)
									\vspace{0mm}
								\end{align}
								\normalsize
								\colr{where, $\lam$ is the optimization variable, $\L\subseteq\mathbb{R}^d$ is a non-empty, closed, and convex set, $\textsf{D}$ is finite, and the objective function separates into node-specific cost functions $D^i$.} The goal is to solve  \eqref{dual_problem2} using only the stochastic { subgradient}s {$\g_t^i(\lam)$} of $D^i(\lb)$. {It is emphasized that the general results presented in this subsection do not required $D^i$ to be differentiable. }As in \eqref{P1}, $\g_t^i(\lam)$ is stochastic due to its dependence on the random variable $\h_t^i$ that is first observed at node $i$ at time $t$. Besides the network resource allocation problem considered here, \eqref{dual_problem2} also arises in the context of machine learning \cite{li2014communication} and distributed parameter estimation \cite{Rabbat}. Before describing the known results related to \eqref{dual_problem2}, the necessary assumptions are first stated. 
								
								\begin{enumerate}
									\item[\textbf{A1.} ] \textbf{Non-expansive projection mapping.} The projection mapping $\px{}$ satisfies $\norm{\px{\x}-\px{\y}} \leq \norm{\x-\y}$ for all $\x$ and $\y$. 
									
									\item[\textbf{A2.} ] \textbf{Zero-mean time-invariant error.} Given $\lb$, the averaged {subgradient} function satisfies $\colr{\partial D^i(\lb)} = \EE[\g_{t}^i(\lb)]$. 
									
									\item[\textbf{A3.} ] \textbf{Bounded moments. }  Given $\lb \in \L$, the second moment of ${\g_t^i(\lam)}$ is bounded as follows:
									\begin{align}
															\EE[\norm{\g_t^i(\lb)}^2] \leq {V_i}^2.
									\end{align}
								\end{enumerate} 
								
								These assumptions are not very restrictive, and hold for most real-world resource allocation problems. A stochastic incremental algorithm for solving \eqref{dual_problem2} was first proposed in \cite{veeravalli}. Given a network with ring topology, the updates in \cite{veeravalli} take the form
								\begin{align}
								\lb_t^i &= \px{\lb_t^{i-1} - \epsilon \g_t^i(\lb_t^{i-1})} \label{vvupdate} 
								\end{align}
								where $\lb_t^0$ is read as $\lb_{t-1}^K$. It was shown in \cite{veeravalli}, that under (\textbf{A1})-(\textbf{A3}), the iterates $\lb^i_t$ are asymptotically near optimal in the following sense\vspace{0mm}
								\begin{align}\label{lemma2_1}
								\underset{t\rightarrow\infty}{\text{liminf}}\ \EE\left[D(\lb_t)\right]\leq\mathsf{D}+\O(\ep).
								\end{align}
								where $\lb_t= \lb_t^0=\lb_{t-1}^K$. Further, for the case when the step size is diminishing, i.e. $\ep_t$ satisfies $\lim\limits_{T\rightarrow\infty}\sum\limits_{t=1}^{T}\ep_t = \infty$ and $\lim\limits_{T\rightarrow\infty}\sum\limits_{t=1}^{T}\ep_t^2 < \infty$, it holds that
								\begin{align}
								\liminf_{t\rightarrow\infty}\Ex{D(\lb_t)}= \mathsf{D}.
								\end{align}
								
								This paper provides the corresponding results for the asynchronous case, where the {subgradient} in \eqref{vvupdate} is replaced by an older copy $\g_{t-\delta_i(t)}^i(\lam_{t-\tau_i(t)}^{i-1})$, that is, the stochastic {subgradient} of $D^i(\lam)$  that depends on the random variable $\h_{t-\delta_i(t)}^i$ and is evaluated at $\lb = \lam_{t-\tau_i(t)}^{i-1}$. The delays satisfy $\tau_i(t) \geq \delta_i(t) \geq 0$ and for the special case of no delay, the stochastic {subgradient} simplifies to $\g_t^i(\lam_{t}^{i-1})$ as in \eqref{vvupdate}. The following additional assumption regarding the delays $\delta_i(t)$ and $\tau_i(t)$ is stated. 
								\begin{enumerate}
									\item[\textbf{A4.} ] \textbf{Bounded delay. } For each $1\leq i \leq K$ and $t\geq 1$, it holds that $0\leq \delta_i(t) \leq \tau_i(t) \leq \tau < \infty$.
								\end{enumerate}
								The boundedness assumption on the delay in \textbf{(A4)} allows us to develop convergence results that hold in the worst case, and has been widely used in the context of asynchronous algorithms \cite{tsitsiklis1986distributed}. It is remarked that an alternative assumption, made in \cite{agarwal2011distributed}, allows the delays $\delta_i(t)$ and $\tau_i(t)$ to be random variables with unbounded supports but finite means, but is not pursued here. Even with bounded delays, the extension to the asynchronous case is not straightforward, since the the old stochastic {subgradient}s are not necessarily descent directions on an average. Indeed, the resulting {subgradient} error at time $t$, defined as
								\begin{align}\label{error_asyn2}
								\e_{t,\delta_i(t)}^i:=&\left[\colr{\partial D^i(\lam_{t}^{i-1})}-\g_{t-\delta_i(t)}^i(\lam_{t-\tau_i(t)}^{i-1})\right]
								\end{align}
								is neither zero-mean nor i.i.d. In other words, the asynchronous algorithm cannot simply be considered as a special case of the inexact {subgradient} method.  
								
It is worth pointing out that there is a subtle difference between the definition of the delayed stochastic {gradient} considered here, and those considered in \cite{adadelay, agarwal2011distributed, feyzmahdavian2015asynchronous}. Specifically, the delayed {gradient} in these works takes the form $\g_t^i(\lam_{t-\tau_i(t)}^{i-1})$ instead of the one in \eqref{error_asyn2}. As a result, given $\lam_{t-\tau_i(t)}^{i-1}$, the {gradient} error at time $t$ in these papers is indeed zero mean and i.i.d., an assumption that simplifies the analysis to a certain extent. It is also remarked that the definition of the delayed stochastic {gradient}s in \cite{sirb2016decentralized} is however similar to that considered here. Different from these works, the dual convergence results developed here consider {subgradient}s instead of {gradient}s, and are therefore applicable to a wider range of problems.

								Within the context of network resource allocation, it is also important to study the (near-)optimality of the allocations $\{\x^i_t, \p^i_t\}$. Towards this end, some additional assumptions are first stated.
								\begin{enumerate}
									\item[\textbf{\textbf{A5}}. ] \textbf{Non-atomic probability density function: } The random variables $\{\bg_t^i\}_{i=1}^K$ have non-atomic probability density functions (pdf).
									
									\item[\textbf{A6.} ] \textbf{Slater's condition: } There exists strictly feasible ${(\tilde{\p}^i,\tilde{\x}^i)}$, i.e., $\EE\left[\sk \g_t^i(\tilde{\p}_t^i,\tilde{\x}^i)\right]> 0$.
									\item[\textbf{A7.} ] \textbf{Lipschitz continuous { gradient}s. } Given $\lb$, $\lb' \in \L$, there exists $L_i < \infty$ such that 
									\begin{align}
									\norm{\nabla D^i(\lb)-\nabla D^i(\lb')} \leq L_i\norm{\lb-\lb'}.
									\end{align} 
								\end{enumerate}
								In (\textbf{A5}), for $\{\bg_t^i\}_{i=1}^K$ to have a non-atomic pdf, it should not have any point masses or delta functions. Note that this requirement is not restrictive for most applications arising in wireless communications; see e.g. \cite{ale10}. The Slater's condition is a standard assumption that ensures that $\mathsf{P} = \mathsf{D}$ and consequently, since $\mathsf{P}$ is finite, so is $\mathsf{D}$. The Lipschitz condition in \textbf{(A7)} is however restrictive, since it requires the dual functions $D^i(\lb)$ to be differentiable with respect to $\lb$. In other words, with \textbf{(A7)}, $\g_t^i(\lb)$ is a stochastic { gradient}, not a {subgradient}. It is remarked however that (\textbf{A7}) always holds if $f^i(\x^i)$ is strongly convex. Moreover, it is generally possible to enforce (\textbf{A7}) artificially by adding a strongly convex regularizer (such as $\theta\norm{\x^i}_2^2$) to the primal objective \cite{chen2016stochastic}. Note however that {\textbf{(A5)}-\textbf{(A7)}} will not be utilized while establishing the dual convergence results. 
								
								The incremental or asynchronous { gradient} methods have thus far never been applied to the problem of network resource allocation. For the classical stochastic dual descent method (cf. \eqref{primal_1}-\eqref{dual}], it is known that under (\textbf{A1})-(\textbf{A3}) and (\textbf{A6}), the average resource allocations $\bar{\x}^i:=\stau\x^i_{t}$ are asymptotically feasible and near-optimal \cite{ale10}.
				\vspace{0mm}
								\subsection{Convergence of the \colr{AIS-SD} algorithm} \label{sec-dual}
								This subsection provides the convergence results for the \colr{AIS-SD} algorithm, applied to \eqref{dual_problem2}. For the general case, the updates take the following form:
								\begin{align}\label{general}
								\lb_t^i &= \px{\lb_t^{i-1} - \ep_t\g^i_{t-\delta_i(t)}(\lb^{i-1}_{t-\tau_i(t)})} & 1\leq i \leq K
								\end{align}
								where $\epsilon_t$ is the step-size, $\g_t^i(\lb)$ is a stochastic {subgradient} of $D^i(\lb)$ and $\lb^0_t$ is read as $\lb^K_{t-1}$. Since the dual problem \eqref{dual_problem} is simply a special case of \eqref{dual_problem2}, the results developed here also apply to the iterates $\{\lb_t^i\}$ generated by Algorithm 1. In order to keep the discussion generic, the results are presented for both, diminishing and constant step sizes. 
								
								\begin{theorem}\label{dual_opt}
									The following results apply to the iterates generated by \eqref{general} {with $\lb_t=\lb_t^0$} under \textbf{(A1)}-\textbf{(A4)}.
									\begin{enumerate}[label=(\alph*)]
										\item \textbf{Diminishing  step-size:} If the positive sequence $\{\ep_t\}$ satisfies $\lim\limits_{T\rightarrow\infty}\sum\limits_{t=1}^{T}\ep_t = \infty$ and $\lim\limits_{T\rightarrow\infty}\sum\limits_{t=1}^{T}\ep_t^2 < \infty$, then it holds that
										\begin{align}\label{decss}
										\liminf_{t\rightarrow\infty}\left[\sum\limits_{i=1}^{K}\Ex{D^i({\lam_{t}})}\right]=\mathsf{D}.
										\end{align}
										\item \textbf{Error bound for constant step size:} For $\ep_t = \ep > 0$, and any arbitrary scalar $\eta > 0$, it holds that
\begin{align}\label{rate_thm_one}
\min_{1\leq t\leq T}\sum_{i=1}^K&\Ex{D^i({\lam_{t}})}\leq \mathsf{D} + {{\frac{\ep C(\tau)+\eta}{2}}}
\end{align}
where $T \leq B^2_0/\ep\eta$. Here, $\tau$ is the maximum delay as defined in (\textbf{A4}), $C(\tau):=C_1+(C_2+\tau C_2')$, $C_1=KV^2$, ${C_2:=2KV^2\frac{K-1}{2}}$, ${C_2':=4K^2V^2}$, and $B_0$ is such that $\norm{{\lb_1}-\lams} \leq B_0$.
										
									\end{enumerate} 
								\end{theorem}
								A popular choice for the diminishing step-size parameter $\ep_t$ required in Theorem \ref{dual_opt}(a) is $\ep_t=t^{-\alpha}$ for $\alpha \in (1/2, 1)$. For this case, the objective function in \eqref{dual_problem2} converges exactly to the dual optimum. On the other hand, with a constant step size $\ep$, the minimum objective value comes to within an $O(\ep)$-sized ball around the optimum as $T \rightarrow \infty$. More precisely, the result in \eqref{rate_thm_one} provides an upper bound on the number of iterations required to come $\eta$-close to this ball. Different from the results in \cite{veeravalli}, the size of the ball now depends on the maximum delay $\tau$, quantifying the worst-case impact of using delayed {subgradient}s.
								
								The proof of Theorem \ref{dual_opt} follows the same overall structure as in \cite{veeravalli}, with appropriate modifications introduced to handle the asynchrony. To begin with, the following intermediate lemma splits a function related to the optimality gap in Theorem \ref{dual_opt} into three different terms, and develops bounds on each. The proof of the following lemma is provided in Appendix~\ref{lemma2_proof}.
								
								\begin{lemma}\label{lemma2} 
									Under \textbf{(A1)}-\textbf{(A4)}, the iterates generated by \eqref{general} {with $\lb_t=\lb_t^0$} satisfy the following bounds:
									\begin{align}\label{dual_asynch}
									{\sum_{t=1}^{T}\sum_{i=1}^{K}}2\ep_t\Ex{D^i({\lam_{t}})}-\mathsf{D}&\leq{B_0^2}+ I_0+I_1
									\end{align}
									where,\vspace{-0mm}
									\begin{align}
									\ \ \ \ I_0 &:= \sum_{t}\bigg(\ep_t^2 KV^2+{2\tau KV^2\sum\limits_{i}\ep_t\ep_{t-\tau_i(t)}}\nonumber \\&\hspace{2.5cm} + 2V^2\sum\limits_{i}(i-1)\ep_t{\ep_{t-\tau_i(t)} }\bigg)\label{i0}\\
									I_1&:=2{\sum_{t=1}^{T}\sum_{i=1}^{K}}\ep_t \Ex{\ip{\g_{t-\delta_i(t)}^i(\lam_{t-\tau_i(t)}^{i-1}), \lam_{t-\tau_i(t)}^{i-1}-\lam_t^{i-1}}} \nonumber\\
									& \hspace{4mm}\leq \col{ 2\tau KV^2{\sum_{t=1}^{T}\sum_{i=1}^{K}} \ep_t \ep_{t-\tau_i(t)}} \label{i1}.
									\end{align}
									\colr{where $\norm{\lb_1^0 - \lams} \leq B_0$. Note that $\lb_1^0=\lb_1$.}
								\end{lemma}
								\normalsize
								Having developed the necessary bounds, the proof of Theorem \ref{dual_opt} is presented next. 
								\begin{IEEEproof}[Proof of Theorem 1] 
									For the positive sequence $\{\ep_t\}$, it holds that
									\begin{align}
									{\sum_{t=1}^{T}\sum_{i=1}^{K}}2\ep_t \Ex{D^i({\lam_{t}})} \geq \left({\min_{1\leq t\leq T}} \sk \Ex{D^i({\lam_{t}})} \right) \sum_{t=1}^T2\ep_t \nonumber.
									\end{align}
									Substituting the bounds obtained in Lemma \ref{lemma2}, and noting that it always holds that {$\ep_t \leq \ep_{t-\tau}$} for all $\tau \geq 0$, we obtain
									\begin{align}\label{dimi_main}
									{\min_{1\leq t\leq T}} & \sk \Ex{D^i({\lam_{t}})} -\mathsf{D}  \nonumber\\
									&\leq \frac{B_0^2+C_1\sum\limits_{t=1}^{T}\ep_t^2  +{\left(C_2+\tau C_2'\right)}\sum\limits_{t=1}^{T}\ep^2_{[t-\tau]_+}}{2\sum\limits_{t=1}^T\ep_t}
									\end{align}
									where, $C_1:= KV^2$, $C_2':=4K^2V^2$ and $C_2:=2KV^2\frac{K-1}{2}$. Note that in \eqref{dimi_main}, we have used the notation $\ep_{[t-\tau]_+}:=\ep_1$ for all $t \leq \tau$.  Next, for the case when $\ep_t$ is diminishing, and satisfies  ${\lim\limits_{T\rightarrow\infty}\sum\limits_{t=1}^{T}\ep_t = \infty}$ and $\lim\limits_{T\rightarrow\infty}\sum\limits_{t=1}^{T} \ep_t^2 < \infty$, the numerator of the bound on the right stays bounded, while the denominator grows to infinity. Consequently, taking the limit of $T \rightarrow \infty$ on both sides of \eqref{dimi_main}, the required result in \eqref{decss} follows.
									Observe that when the step size is constant, the bound in \eqref{dimi_main} can be written as
									\begin{align}\label{rate_prev}
									{\min_{1\leq t\leq T}} \sk \Ex{D^i({\lam_{t}})}-\mathsf{D}
									&\leq\frac{B_0^2}{2\ep T}+{\frac{\ep}{2}} C_1+{\frac{\ep}{2}} ({C_2+\tau C'_2})\nonumber\\
									& \leq {\frac{B_0^2}{2\ep T}+{\frac{\ep}{2}} C(\tau)} 
									\end{align}
									\normalsize	 
									where $C(\tau)$ is as defined in Theorem \ref{dual_opt}. In the limit as $T\rightarrow\infty$, the bound becomes
										\begin{align}\label{rate_one}
									\inf_{t\geq 1}\Ex{D(\lam_t)}\leq \textsf{D}+\frac{\ep C(\tau)}{2}.
									\end{align}
									which is the asymptotic version of the result in \eqref{rate_thm_one}. 
									
									The rate result in \eqref{rate_thm_one} builds upon a similar result from \cite[Prop. 3.3]{nedic2001convergence}. Intuitively, $\Ex{D(\lam_t)}$ continues to decrease as long as it is significantly larger than $\mathsf{D}$. The rest of the proof characterizes the resulting decrement rigorously and subsequently invokes the monotone convergence theorem in order to establish that $\Ex{D(\lam_t)}$ must eventually come close to $\mathsf{D}$. Given arbitrary $\eta >0$ and recalling that $\lam_t := \lam_t^0$, define the sequence 
	\vspace{0mm}
	\begin{align}
	\mathring{\lam}_{t+1}:=\begin{cases}
	\lam_{t+1}\ ; \ \text{if} \ \ \Ex{D(\lam_t)}\geq \textsf{D}+\frac{\ep C(\tau)+\eta}{2}\\
	\lam^\star \ ; \ \text{otherwise}.
	\end{cases}
	\end{align}	
	Alternatively, $\mathring{\lam}_t$ is same as $\lam_t$ until $\lam_t$ enters level set defined as
	\vspace{0mm}
	\begin{align}
	L=\left\{\lam\in\Lambda ~|~ \Ex{D(\lam)}< \textsf{D}+\frac{\ep C(\tau)+\eta}{2}\right\}
	\end{align}
	and $\mathring{\lam}_t$ terminates at $\lam^\star$. From \eqref{lbexp3} and Lemma~\ref{lemma2}, we have for constant step size $\ep_t=\ep$ that
	\begin{align}
	\Ex{\norm{\mathring{\lam}_{t+1}-\lam^\star}^2}\leq& \Ex{\norm{\mathring{\lam}_{t}-\lam^\star}^2}-2\ep\left[\Ex{D(\mathring{\lam}_t)-\textsf{D}}\right]\nonumber\\&+\ep^2C(\tau)\label{rateproof1}
	\end{align}
	where $\mathring{\lam}_{t+1}:=\mathring{\lam}_{t}^K$ and $\mathring{\lam}_t=\mathring{\lam}_{t}^0$. Next define
	\begin{align}
	z_t:=\begin{cases} 2\ep\left[\Ex{D(\mathring{\lam}_t)-\textsf{D}}\right]-\ep^2C(\tau)\ \ \text{if}\ \ \mathring{\lam}_t\notin L\\
	0\ \ \text{if}\ \ \text{otherwise},
	\end{cases}
	\end{align} 
	so that \eqref{rateproof1} can be written as
	\begin{align}
	\Ex{\norm{\mathring{\lam}_{t+1}-{\lam}^\star}^2}\leq \Ex{\norm{\mathring{\lam}_{t}-\lam^\star}^2}-z_t.\label{moton}
	\end{align}\normalsize 
	From the Monotone convergence theorem, we have that $\sum\limits_{t=1}^{\infty}z_t<\infty$, implying that there exists $T < \infty$, such that $z_t=0$ for all $t\geq T$. Observe that for the case when $\mathring{\lam}_t\notin L$, it holds that 
	\begin{align}
	z_t=&2\ep\left[\Ex{D(\lam_t)-\textsf{D}}\right]-\ep^2C(\tau)\\
	\geq&2\ep\left[\textsf{D}+\frac{\ep C(\tau)+\eta}{2}-\textsf{D}\right]-\ep^2C(\tau)
	\geq \ep\eta.
	\end{align} 
	Consequently, it follows from \eqref{moton} that\vspace{0mm}
	\begin{align}
	\mathbb{E}[||\mathring{\lam}_{T+1}-\lam^\star||^2]&\leq ||\mathring{\lam}_{1}-\lam^\star||^2-\sum\limits_{t=1}^{T}z_t \\
	&\leq B_0 - \sum_{t=1}^Tz_t.
	\end{align}\normalsize	
	Since the term on the left is non-negative, we have that
	$B_0 \geq \sum\limits_{t=1}^{T}z_t\geq T \ep\eta$, yielding the required bound on $T$.
\end{IEEEproof}

\subsection{ Primal near optimality and feasibility}\label{sec-pri}
\colr{The AIS-SD algorithm of Sec. \ref{sec-dual}, when applied to solve the dual problem in \eqref{dual_problem}, is referred to as the AIS-DD algorithm. In order to ensure that the results developed thus far continue to apply to the dual problem, assumptions (\textbf{A5})-(\textbf{A7}) are also required. As mentioned earlier, for the primal problem, $\L$ is simply the non-negative orthant implying that $\mathsf{D}$ is finite.} This subsection establishes the average near-optimality of the \colr{AIS-DD} algorithm in \eqref{primal_sub}-\eqref{dual_sub}. Note that Theorem \ref{dual_opt} does not imply that the allocations $\{\x_t^i,\p_t^i\}$ converge. Instead, the results will make use of the ergodic limit variable 
								\begin{align}
								{\bar{\x}^i_T:=\frac{1}{T}\sum_{t=1}^{T}\x^i_t}
								\end{align} 
								for each $1 \leq i \leq K$. The main theorem for this subsection is presented next.
								\begin{theorem}\label{primal_opt} 
									Under \textbf{(A1)}-\textbf{(A7)} {and for constant step size $\ep>0$}, the iterates generated by \eqref{primal_sub}-\eqref{dual_sub} follow:
									\begin{enumerate}[label=\Alph*.]
										\item \textup{\textbf{Primal near optimality}}\begin{align}\label{Primal_avg}
										{\liminf\limits_{T\rightarrow\infty}}\sk \Ex{f^i\left(\bar{\x}^i_T\right)}\geq & \mathsf{P}-\ep({C_3+\tau C_4}) 
										\end{align}
										{where},\vspace{-4mm}
										\begin{align}
										C_3 =& {(\col{V^2K^2})/2}\nonumber\\
										C_4=& {K^2BLV+2K^2V^2}\nonumber.
										\end{align}
										\item\textbf{\textup{ Asymptotic feasibility}} 
										\begin{align} \label{asyfeas}
										\liminf\limits_{T\rightarrow\infty}{\frac{1}{T}}{\sum_{t=1}^{T}\sum_{i=1}^{K}}\EE\left[\g_{t-\tau_i(t)}^i(\lb^{i-1}_{t-\tau_i(t)})\right]\succeq 0. \ \
										\end{align}
									\end{enumerate}
								\end{theorem}
								
								Intuitively, the resource allocations in \eqref{primal_sub} are near-optimal, with optimality gap depending on the step size $\ep$ and the delay bound $\tau$. Further, the allocations are almost surely asymptotically feasible, regardless of the the delay bound or the step size. As in Sec. \ref{sec-dual}, the proof of Theorem \ref{primal_opt} proceeds by first splitting the optimality gap into three terms and developing bounds on each. The required results are summarized into the following intermediate Lemmas, whose proofs are deferred to Appendices~\ref{boundedness} and \ref{lemma3proof} respectively.  
								\begin{lemma}\label{dual_B}
									\colr{Under \textbf{(A1)}-\textbf{(A6)}}, the iterates $\lb_t^i$ obtained from \eqref{dual_sub} are bounded on an average, i.e., there exists $B < \infty$ such that $\EE\norm{\lb_t} \leq B$ for all $t \geq 1$. 
								\end{lemma}
							\begin{lemma}\label{lemma3} 
									\colr{Under \textbf{(A1)}-\textbf{(A7)}}, the iterates generated by \eqref{primal_sub}-\eqref{dual_sub} satisfy the following bounds:
										\begin{align}
									\sk \Ex{f^i(\bar{\x}_T^i)} \geq \mathsf{D}-	I_2 - I_3
									\end{align}
									where,\vspace{0mm}
									\begin{align}
									\hspace{0mm}I_2 :=& {{\frac{1}{T}}{\sum_{t=1}^{T}\sum_{i=1}^{K}}\E\left[D^i({\lb_{t}})-D^i(\lamb)\right]} \nonumber
									\\
									&\leq {\ep V^2K(K-1)/2 }\nonumber\\
									I_3 :=& {\frac{1}{T}}{\sum_{t=1}^{T}\sum_{i=1}^{K}} \Ex{\ip{\lamb,\nabla D^i(\lamb)}}\nonumber
									\\
									& \leq \frac{\norm{\lb_1}^2}{2\ep T}+\frac{\ep KV^2 }{2} + I_4 \nonumber\\
									I_4 :=&  {\frac{1}{T}}{\sum_{t=1}^{T}\sum_{i=1}^{K}}\Ex{\ip{\lamb,(\nabla D^i(\lamb)-\g_{t-\delta_i(t)}^i(\lam_{t-\tau_i(t)}^{i-1}))}}\nonumber\\
									&\leq  {\ep\tau K^2V(BL+2V)} \nonumber.
									\end{align}
								\end{lemma}
								\normalsize
								Having established the intermediate results, the proof of Theorem \ref{primal_opt} is now presented. 
								
								\begin{proof}[Proof of Theorem~\ref{primal_opt}]
									The primal near-optimality can be established directly from Lemma \ref{lemma3}. Specifically, summing the bounds for $I_2$, $I_3$, and $I_4$, and taking the limit as $T \rightarrow \infty$, the bound in \eqref{Primal_avg} follows.
									
									In order to establish \eqref{asyfeas}, observe that for any $t\geq 1$ and $1 \leq i \leq K$, it holds that
									\begin{align}
									\lb^{i}_{t} &= \left[\!\lb^{i-1}_{t}\!\!-\!\ep  \g_{t-\delta_i(t)}^i(\lb^{i-1}_{t-\tau_i(t)})\right]^+\nonumber\\
									 &\succeq \lb^{i-1}_{t}-\ep   \g_{t-\delta_i(t)}^i(\lb^{i-1}_{t-\tau_i(t)}) \label{lbsum}
									\end{align}
									\normalsize
									where the inequality holds element-wise. Summing both sides over all $1\leq t \leq T$ and $1 \leq i \leq K$, and rearranging, it follows that
									\begin{align}
									{\frac{1}{T}}{\sum_{t=1}^{T}\sum_{i=1}^{K}}\g_{t-\delta_i(t)}^i(\lb^{i-1}_{t-\tau_i(t)}) &\succeq  \frac{1}{\ep T} {\sum_{t=1}^{T}\sum_{i=1}^{K}} (\lb_t^{i-1} - \lb_t^i)\nonumber
									\\
									&\succeq \frac{\lb^{K}_{1}-\lb^{K}_{t+1}}{\ep T} \nonumber.
									\end{align}
									\normalsize
									Finally, since $\lb_1^K \succeq 0$, \colr{taking expectations on both sides}, it follows that
									\begin{align}\label{pre}
									{\frac{1}{T}}{\sum_{t=1}^{T}\sum_{i=1}^{K}}\EE\left[\g_{t-\delta_i(t)}^i(\lb^{i-1}_{t-\tau_i(t)})\right] \succeq -\frac{B}{\ep T}
									\end{align}
									\normalsize
									{where \eqref{pre} holds due to Lemma~\ref{dual_B}}. In other words, given any $\alpha >0$, there exists $t_0 \in \mathbb{N}$ such that for all $T \geq t_0$, 
									\begin{align}
									{\frac{1}{T}}{\sum_{t=1}^{T}\sum_{i=1}^{K}}\EE\left[\g_{t-\delta_i(t)}^i(\lb^{i-1}_{t-\tau_i(t)})\right] \succeq -\alpha. \ \ 
									\end{align}
									Taking the limit as $T\rightarrow \infty$, the result in \eqref{asyfeas} follows.
								\end{proof}
								\section{Application to co-ordinated beamforming}\label{beam}
								This section considers the co-ordinated downlink beamforming problem in wireless communication networks. The usefulness of the proposed stochastic incremental algorithm is demonstrated by applying it to the beamforming problem and solving it in a distributed and online fashion. Simulations are carried out to confirm that the performance of the proposed algorithm is close to that of the centralized algorithm. 
												\vspace{0mm}
								\subsection{Problem formulation}  
								Consider a multi-cell multi-user wireless network with $B$ base stations and $U$ users. Each user $j \in \{1, \ldots, U\}$ is associated with a single base station $b(j) \in \{1, \ldots, B\}$, and the set of users associated with a base station $i$ is denoted by $\U_i:=\{j | b(j) = i\}$. For the sake of consistency, this section will utilize indices $i$ and $m$ for base stations, and indices $j$, $k$, and $n$ for users, with the additional restriction that $b(j) = b(k) = i$ and $b(n) = m$. Within the downlink scenario considered here, user $j$ can only receive data symbols $s_{j} \in \mathbb{C}$ from its associated base station $b(j)$. The signals transmitted by the base station $i$ intended for other users $k \in \U_i\setminus \{j\}$, as well as the signals transmitted by other base stations $m \neq i$ constitute, respectively, the intra-cell and inter-cell interference at user $j$. The base station $i$, equipped with $N_i$ transmit antennas, utilizes the transmit beamforming vector $\w_{j} \in \mathbb{C}^{N_i \times 1}$ for each of its associated user $j \in \U_i$. Consequently, the received signal at user $j$ is given by
								\begin{align}
								y_j = \h_{ij}^H(\w_{j}s_{j} + \hspace{-0.3cm}\sum_{k \in \U_i\setminus \{j\}} \hspace{-0.3cm}\w_{k}s_{k}) + \sum_{\substack{m \neq i \\ n \in \U_m}} \hspace{-0.1cm}\h_{mj}^H\w_{n}s_{n} + e_j \nonumber
								\end{align}
								where $\h_{ij}$ denotes the complex channel gain vector between base station $i$ and user $j$, and $e_j$ is the zero mean, complex Gaussian random variable with variance $\sigma^2$ that models the noise at user $j$. Assuming $s_k$ to be independent, zero-mean, and with unit variance, the expression for the signal-to-interference-plus-noise ratio (SINR) at user $j$ is given by
								\begin{align}\label{sinr}
								\text{SINR}_{j} := \frac{\abs{\h_{ij}^H\w_{j}}^2}{\hspace{-0.2cm}\sum\limits_{k \in \U_i\setminus \{j\}} \hspace{-0.3cm}\abs{\h_{ij}^H\w_{k}}^2 + \sum\limits_{m \neq i}\sum\limits_{n \in \U_m} \hspace{-0.1cm}\abs{\h_{mj}^H\w_{n}}^2 + \sigma^2}
								\end{align}
								where $i = b(j)$ is the associated base station. 
								
								Within the classical co-ordinated beamforming framework, the goal is to design the beamformers $\{\w_{j}\}_{j=1}^U$ so as to minimize the transmit power, while meeting the SINR constraints at each user. The required optimization problem becomes \cite{4558565}
								\vspace{0mm}
								\begin{align}
								\min_{\{\w_{j}\}_{j=1}^U} &\sum_{j=1}^U \norm{\w_j}^2 \ \ \ \text{subject to }  \text{SINR}_j \geq {\gamma_j}  ~~ \forall~j\label{classic} 
								\end{align}
								where ${\gamma_j}$ is a pre-specified quality-of-service (QoS) threshold {for user $j$}. While the beamforming vectors resulting from \eqref{classic} are optimal, the centralized nature of the optimization problem renders it impractical for application to real networks. For instance, the solution proposed in {\cite{4558565}} requires the estimated channel gains $\{\h_{ij}\}$ to be collected at a centralized location, where \eqref{classic} is solved via an iterative algorithm. In practice however, the entire parameter exchange and the algorithm must complete within a fraction of the coherence time of the channel, lest the designed beamformer becomes obsolete. Such a solution is therefore difficult to implement, not robust to node or link failures, and not scalable to large networks. 
								
								Observe that the {modified} version of \eqref{classic} can be written as
								\vspace{0mm}
								\begin{subequations}\label{dist1}
									\begin{align}
									&\min_{\{\w_{j}, I_{nj}\}} \sum_{j=1}^U \norm{\w_j}^2 \ \ \ \text{subject to }&  \\
									&\frac{\abs{\h_{ij}^H\w_{j}}^2}{\hspace{-0.2cm}\sum\limits_{k \in \U_i\setminus \{j\}} \hspace{-0.3cm}\abs{\h_{ij}^H\w_{k}}^2 + \hspace{-0.1cm}{I_{j}^2} + \sigma^2} \geq {\gamma_j} & \forall~j\label{sinrdist1}\\
									&\col{\sum\limits_{m \neq i}\sum\limits_{n \in \U_m} \hspace{-0.1cm}\abs{\h_{mj}^H\w_{n}}} \leq \col{I_{j}} \ \ & \hspace{-2cm}\forall~j \label{inj}
									\end{align}
								\end{subequations}
								\normalsize
								where $i = b(j)$. Note that constraints in \eqref{sinrdist1} and \eqref{inj} will ensure the SINR is still greater than the required threshold of $\gamma_j$. It is due to the fact that the feasible set is restricted and feasible set of \eqref{dist1} will be subset of that of \eqref{classic} and solution found for \eqref{dist1} can be used for \eqref{classic}. Next, the use of primal or dual decomposition techniques can yield a distributed algorithm for \eqref{dist1}. Nevertheless, such distributed algorithms also suffer from the limitations mentioned earlier, since the optimum beamforming vectors are required at every time slot. 
								
								On the other hand, within the uncoordinated beamforming framework, the optimization variable {$I_{j}$} in \eqref{dist1} is replaced with a pre-specified threshold $\rho$. This renders \eqref{dist1} separable at each base station, allowing beamforming vectors to be designed in parallel. However, the resulting beamformers are suboptimal, and may even render the problem infeasible if $\rho$ is too small or too large. 
								\vspace{0mm}
								\begin{subequations}\label{dist2}
									\begin{align}
									&\min_{\{\w_{j}, I_{nj}\}} \sum_{j=1}^U \norm{\w_j}^2 \ \ \ \text{subject to }&  \\
									&\frac{\abs{\h_{ij}^H\w_{j}}^2}{\hspace{-0.2cm}\sum\limits_{k \in \U_i\setminus \{j\}} \hspace{-0.3cm}\abs{\h_{ij}^H\w_{k}}^2 + \hspace{-0.1cm}\rho^2\sum\limits_{m \neq i}card{(\U_m)} + \sigma^2} \geq {\gamma_j} \ \ ; \ \ \forall~j\label{sinrdist2}\\
									&{ {\abs{\h_{mj}^H\w_{n}(t)} \leq \rho}  \hspace{2.3cm}m\neq i, n \in U_m, \forall~j. } \label{inj3}
									\end{align}
								\end{subequations}
									
									A compromise is possible within the stochastic optimization framework by requiring the bound in \eqref{inj} to only be satisfied on an average. Note that this amounts to relaxing the optimization problem \eqref{dist1} since the SINR constraint is no longer binding at every time slot. The overall stochastic optimization problem can be expressed as \vspace{0mm}
									\begin{subequations}\label{problem2}
										\begin{align}
										&\min_{\{\w_{j}(t), I_{nj}(t)\}} \sum_{j=1}^U \norm{\w_j(t)}^2 \ \ \ \text{subject to }  \\
										&\frac{\abs{\h_{ij}^H\w_{j}(t)}^2}{\hspace{-0.2cm}\sum\limits_{k \in \U_i\setminus \{j\}} \hspace{-0.3cm}\abs{\h_{ij}^H\w_{k}(t)}^2 + \hspace{-0.1cm}{I_{j}^2(t)} + \sigma^2} \geq {\gamma_j}  \ \ \ \ \forall~j \label{SINR_const}\\
										&{\col{\sum\limits_{m \neq i}\sum\limits_{n \in \U_m} \EE\abs{\h_{mj}^H\w_{n}(t)}} } \leq \Ex{I_{j}(t)} \hspace{.8cm}\forall~j \label{intercell_inr_cons}\\
										& \col{\abs{\h_{mj}^H\w_{n}(t)} \leq \rho}  \hspace{2.3cm}m\neq i, n \in U_m, \forall~j \label{inj2} 
										\end{align}
									\end{subequations}
									
									where $i = b(j)$. Different from \eqref{classic} or \eqref{dist1}, the stochastic optimization problem \eqref{problem2} involves finding policies $\w_j(t)$ and ${I_{j}(t)}$, which are not necessarily optimal for every time slot $t$, but only on an average. Specifically, the intercell interference is bounded on an average [cf. \eqref{intercell_inr_cons}], but also instantaneously [cf. \eqref{inj2}], so as to limit the worst case SINR. The problem in \eqref{problem2} can be readily implemented using the proposed  distributed and asynchronous stochastic dual descent algorithm. In contrast to \eqref{dist1}, the stochastic algorithm is not required to converge at every time slot, and allows cooperation over heterogeneous nodes. 
									\vspace{0mm}
													\begin{algorithm}
														\caption{:Operation at node $i$ }\label{algo_1}
														\begin{algorithmic}[1]
															\State {\bf Set} $t=1$, \col{initialize $\lambda^{0}_{j}(1)$}
															\State \col{{\textbf{Beamformer design:}}  At the start of time instant $t$, given old dual variables ${\{\lambda_{j}^i(t-\pi_i(t))\}}$ 
															\begin{align}
															&\hspace{-1cm}\{\w_{j}(t), I_{j}(t)\}_{\{\forall j|i=b(j)\}} :=\label{primup4}\\& \argmin_{\w,I}\sum\limits_{j\in\mathcal{U}_i}{\left[\norm{\w_{j}}^2-\lambda_{j}^i(t-\pi_i(t))I_j\right]} \nonumber\\
															&\ \ \ \ \ \ \ \  +\sum\limits_{j\notin\mathcal{U}_i}{\left[\lambda_{j}^i(t-\pi_i(t))\left(\sum\limits_{n \in \U_i} \hspace{-0.1cm}\abs{\h_{ij}(t)^H\w_{n}}\right)\right]} \nonumber\\
															&\hspace{2cm}\text{subject to} \ \ \eqref{SINR_const}, \eqref{inj2}	\nonumber
															\end{align}\normalsize}
														\vspace{0mm}
															\State \col{\textbf{(Optional) Receive : } \col{$\{\lambda_{j}^i(t')\}_{\forall j}$, where $t'\leq t$}}
															\State \col{\textbf{Dual update:} To complete cycle $t'$, {$\text{for all} \ j$}
																\begin{itemize}
																	\item If $b(j)=i$, then
																	\begin{align}
																	\lambda_{j}^{i+1}(t')=&\left[\lambda_{j}^{i}(t')-\ep I_{j}(t-\delta_i(t))\right]^+\nonumber
																	\end{align}
																	\item If $b(j)\neq i$, then
																	\begin{align}\label{dual_cor}
																	\hspace{-0.9cm}\lambda_{j}^{i+1}(t')=&\left[\!\lambda_{j}^{i}(t')\!+\!\ep\left(\sum\limits_{n \in \U_i} \hspace{-0.1cm}\abs{\h_{ij}(t\!-\!\delta_i(t))^H\w_{n}(t-\delta_i(t))}\right)\right]^+
																	\end{align}
																\end{itemize}
																where, $\lambda_{j}^{1}(t'+1) =\lambda_{j}^{B+1}(t') $ for all $j$}
															\State \textbf{Set} $t=t+1$, go to step 2
														\end{algorithmic}
													\end{algorithm}\vspace{0mm}
									\subsection{Solution to optimization problem}\label{numsol}
									The \colr{AIS-DD} algorithm proposed in Sec. \ref{sec-inc} can now be applied to solve \eqref{problem2}. To this end, associate dual variables ${\lambda_{j}}$ for all {users} ${\{j \in \left[1,U\right]\}}$, and observe that the primal variables at node $i$ include $\{\w_j\}_{j \in \U_i}$ and ${\{I_{j}\}_{j \in \U_i}}$. Departing from the notational convention used thus far, the subscript in $\lambda_j$ is used for indexing the users, while time dependence is indicated by $\lambda_j(t)$. Proceeding as in Sec. \ref{sec-inc}, and recalling that the indices $j$ and $n$ are such that $i = b(j) \neq b(n) = m$, the operation at node $i$ is summarized in Algo.~\ref{algo_1}. Observe that such an implementation entails allocating resources prior to the dual updates, and thus results in the delay of at least one, i.e., $\pi_i(t) \geq 1$, compared to the synchronous version. Conversely, the dual updates occur as and when they are passed around, without creating a bottleneck for the resource allocation. For the sake of simplicity, it is assumed that the dual updates occur along the route $1$, $2$, $\ldots$, $K$.

									Next, simulations are carried out demonstrate the applicability of the stochastic algorithm to the beamforming problem at hand. For the simulations, we consider a system with \col{$B = 10$ and $U = 10$}, with one users per cell. Each of the  base stations  have ten antennas ($N_i = 10$), while  the other algorithm parameters are $\col{\ep=0.5}$, $\sigma^2=1$, $\rho=1.65$ ${\gamma_j = 10\ \text{dB} \ \text{for all}\  j}$. In order to keep the simulations realistic, we assume that the delays in the dual updates arise from random events such as node and link failures. For the centralized algorithm, a random subset of four out of ten nodes are selected to transmit their current { gradient}s to the FC at every time slot. Since the FC utilizes old { gradient}s for the other nodes, it results in an average delay of 5.4 time slots. Similarly, for the incremental algorithm, it is assumed that at every time slot, five to fifteen dual update steps (cf. \eqref{dual}) occur, resulting in an average delay of 5.4 time slots. For instance, if at any time slot, only 8 nodes update, it will result in a delay {$\pi_i(t) = 2$} at the remaining two nodes, where the dual update will occur at the start of the next time slot. The delay may increase further if fewer than 10 nodes update for consecutive time slots and conversely, may decrease if more than 10 updates occur per time slot. Fig. \ref{one} shows the running average of the primal objective function as a function of time \col{using Monte Carlo simulations}. For comparison, the performance of the classical centralized stochastic {gradient} method [cf. \eqref{primal_1}-\eqref{dual}], assuming  perfect message passing, is also shown. As evident, the performance loss due to the delays in the availability of the dual variables in minimal. 
										
										\begin{figure}	
											\hspace{-1mm}\includegraphics[width =1\columnwidth,height=0.18\textheight]{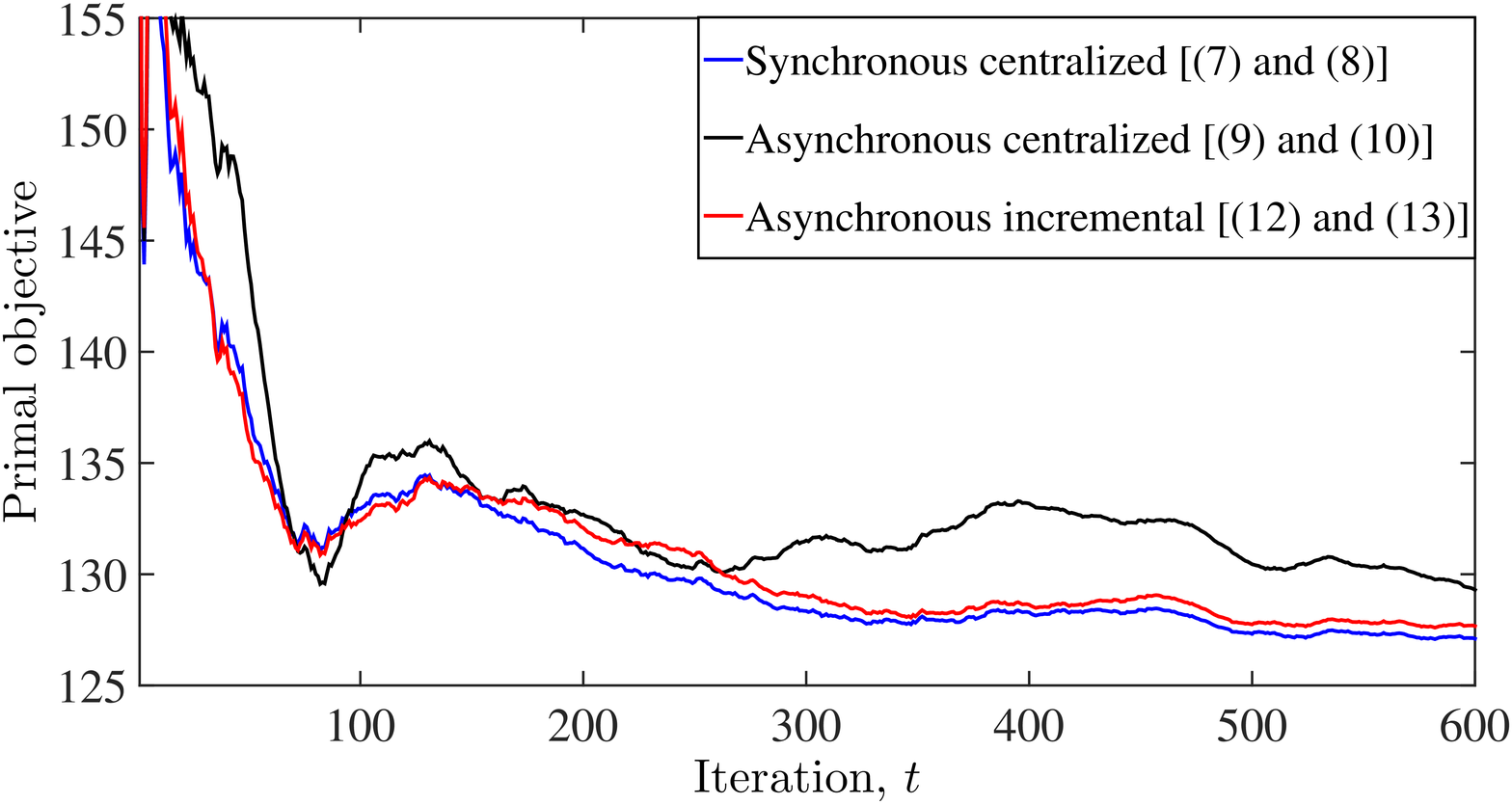}	
											\vspace{0mm}	
											\caption{Primal objective function against {iteration index} $t$ for $B=10$.}
											\label{one}	
											\vspace{-5mm}	
										\end{figure}
									
									In order to motivate the stochastic formulation over the deterministic one, Fig. \ref{two} also compares the average transmit power and SINR achieved for the various cases and for different values of the parameter $\rho$. As expected, the distributed deterministic algorithm performs poorly since it forces the SINR bound to be a constant that does not depend on the channel. By design, the worst case SINR is bounded below by one at every time slot in both the deterministic formulations. Interestingly, the worst case SINR achieved for the relaxed stochastic formulation is also close to one on an average. In return, the stochastic algorithm yields an average transmit power that is equal to or below that obtained by the centralized deterministic formulation. In other words, it is always possible to artificially raise $\gamma$ to a value that is slightly higher than one, so as to obtain an average SINR above $10$ dB, while still getting near-optimal average transmit power. 
								\begin{figure}			
						\hspace{-8mm}\includegraphics[width =1.1\columnwidth,height=0.3\textheight]{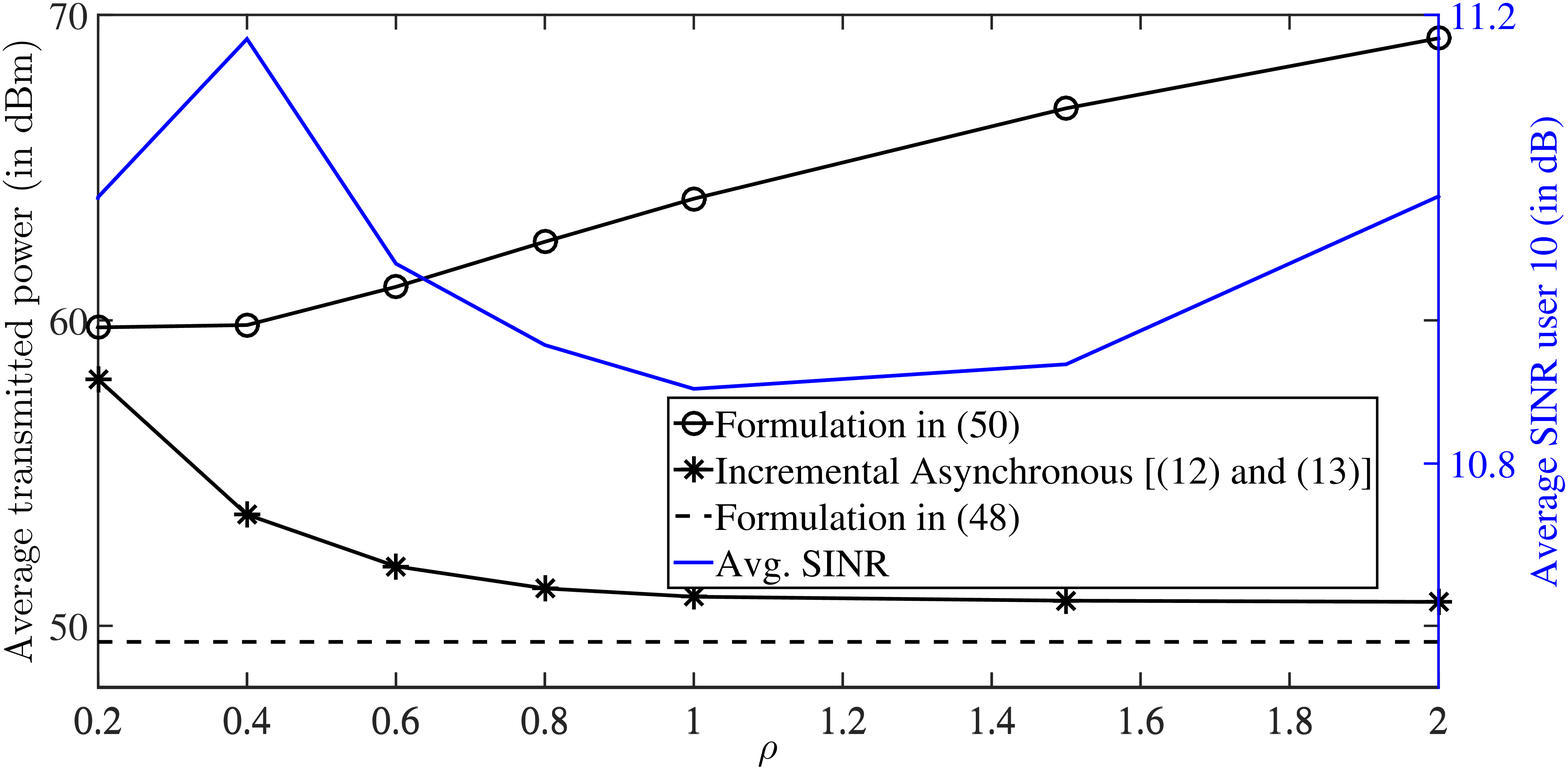}	\vspace{-3mm}	
						\caption{Transmitted power and SINR against $\rho$.}
						\label{two}
						\vspace{-5mm}
					\end{figure}
					
					\begin{figure}	
						\hspace{-5mm}\includegraphics[width =1.1\columnwidth,height=0.3\textheight]{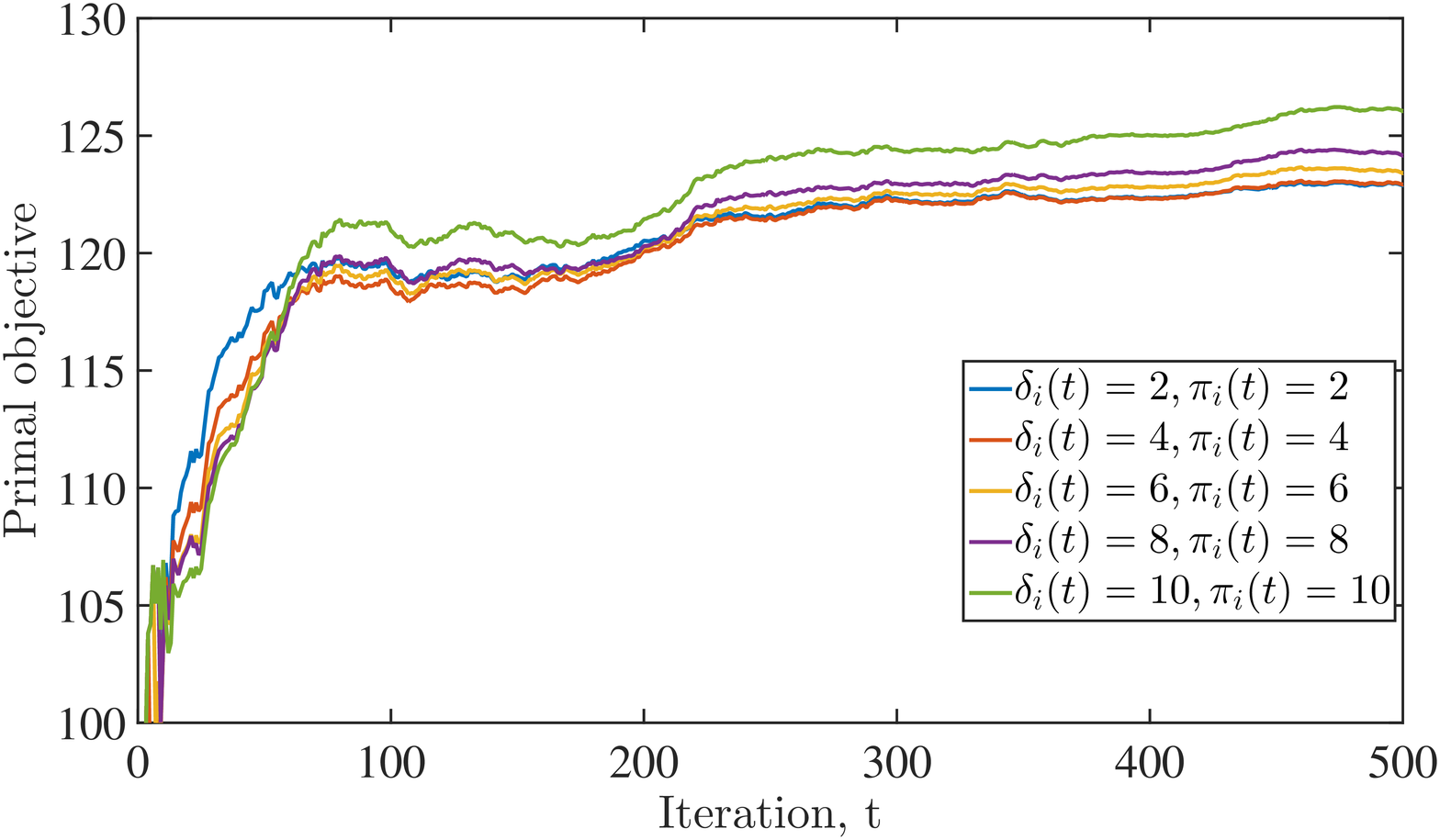}		
						\caption{Primal objective function against time $t$ for different delays,  $B=10$.}
						\label{delay_paper}	
					\end{figure}		
Next, we study the effect of delay on the rate of convergence of the AIS-DD algorithm. For this case, a simple system with $B = 10$ and $U=10$, and constant delays at all the nodes is considered. The base stations have ten antennas each ($N_i=10$) and  the other algorithm parameters are $\ep=0.2$, $\sigma^2=1$, and $\rho=1.65$. Fig. \ref{delay_paper} shows the evolution of the primal objective function for various delay values. As expected, the convergence is slower if both $\pi_i(t)$ and $\delta_i(t)$ are consistently larger. Interestingly however, a small increase in the delays amounts to only a marginal loss in performance.

Finally, in order to demonstrate the scalability of the proposed algorithm, Fig. \ref{Large_nodes} shows an example run for a system with $B=50$ nodes and {$U =50$}. The  base stations have ten antennas each ($N_i= 10$), while  the other algorithm parameters are ${\ep=0.5}$, $\sigma^2=1$, and $\rho=5$. The delay is generated in the similar manner as for the earlier simulations. It can be observed that even when the number of nodes is large, the difference between the performance of the synchronous and asynchronous algorithms remains relatively small.

\begin{figure}	
	\centering
	\includegraphics[width =1\columnwidth,height=0.3\textheight]{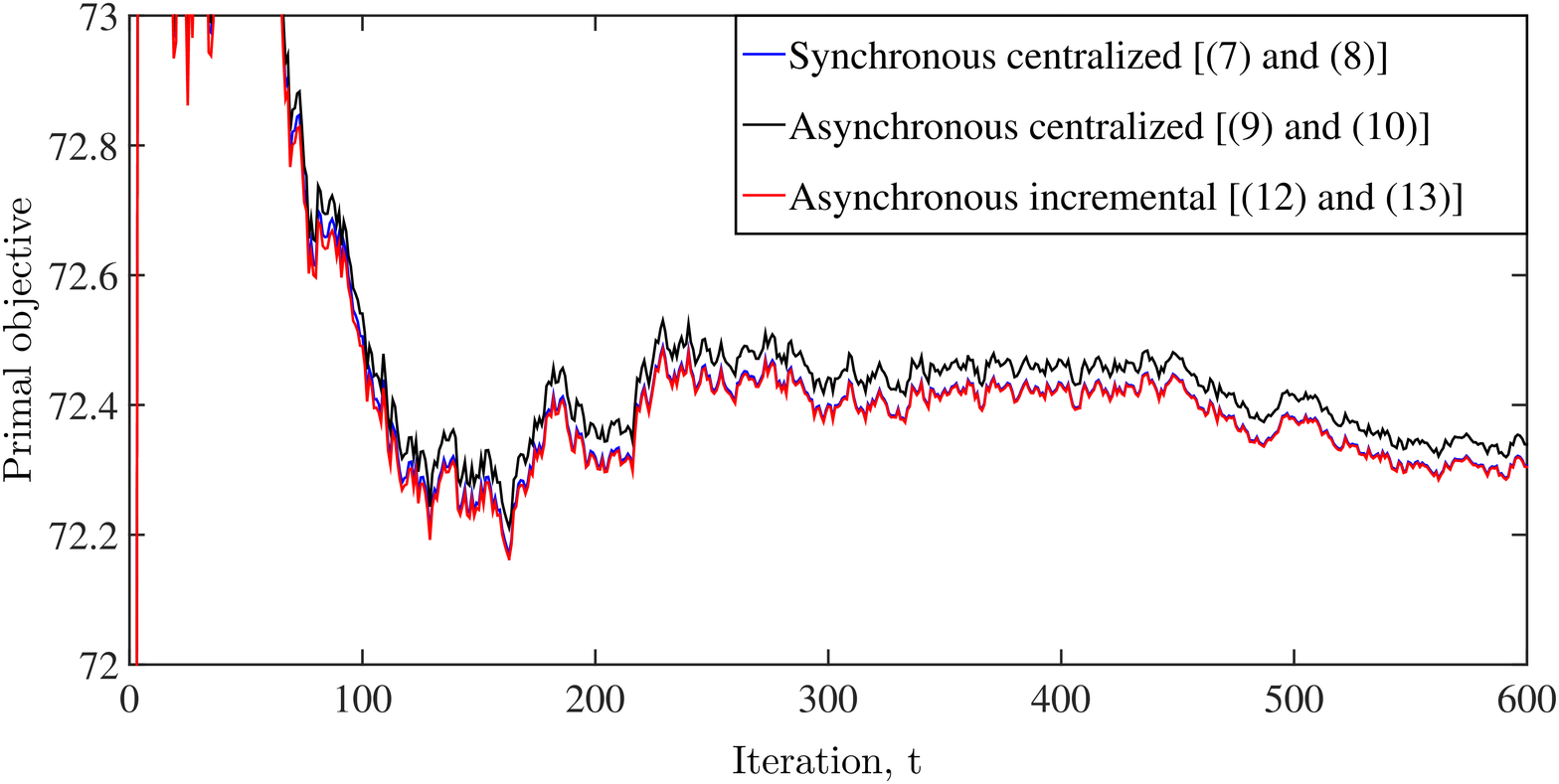}		
	\caption{Primal objective function against {iteration index} $t$, $ B=50$.}
	\label{Large_nodes}	
	\vspace{-8mm}	
\end{figure}


								\vspace{-4mm}
									\section{Conclusion}\label{conc}
									 This paper considers a constrained stochastic resource allocation problem over a heterogeneous network. An asynchronous incremental stochastic dual descent method is proposed for solving the same. The proposed algorithm utilizes delayed {gradients} for carrying out the updates, resulting in an attractive feature that allows nodes to skip or postpone some updates. The convergence of the proposed algorithm is established for both constant and diminishing step sizes. Further, it is shown that the resource allocations arising from the proposed algorithm are also asymptotically near-optimal. A novel multi-cell coordinated beamforming problem is formulated within the stochastic framework considered here, and solved via the proposed algorithm. Simulation results reveal that the impact of using stale stochastic {gradients} is minimal. 
\appendices						
\section{Proof of Lemma~\ref{lemma2}}\label{lemma2_proof}
									
\subsubsection{Preliminaries} Before deriving the required bounds, some preliminary results are first obtained. Recall that the quantity $\g_{\upsilon}^i(\lam)$ denotes the stochastic {(sub-)gradient} of $D^i(\lam)$ at time $t = \upsilon$ and evaluated at $\lam$. Within the context of the dual descent algorithm, we also have that $\g_{\upsilon}^i(\lam):=\g_{\upsilon}^i(\p^i_{\upsilon}(\lam),\x_{\upsilon}^i(\lam))$ for $\upsilon \geq 1$ and all $1\leq i\leq K$. In particular the updates in \eqref{general} use $\upsilon = t-\delta_i(t)$ and $\lam = \lam_{t-\tau_i(t)}^{i-1}$, where $\tau_i(t) = \delta_i(t)+\pi_i(t)$. It can be seen that for the special case of the synchronous algorithm, we have that $\tau_i(t) = \delta_i(t) = 0$ and the stochastic {(sub-)gradient} is written as $\g_t^i(\lam_t^{i-1})$.


For the sake of convenience, let us denote $\g_{t-\delta_i(t)}^i := \g_{t-\delta_i(t)}^i(\lb_{t-\tau_i(t)}^{i-1})$. First, we establish that the distance between the iterates $\lb_t^i$ and $\lb_{{\ell}}^i$ is bounded by a term that is proportional to the step size. From the updates in \eqref{general}, it holds for all $t$ and $1\leq i \leq K$ that
									\begin{align}
									\EE\norm{\lb_t^i-\lb_t^{i-1}} &= \EE\norm{\px{\lam_t^{i-1} -\ep_t\g_{t-\delta_i(t)}^i} - \lb_t^{i-1}} \nonumber\\
									&\leq \ep_t\EE\norm{\g_{t-\delta_i(t)}^i} \leq \ep_tV_i \label{lbbound}
									\end{align}
									where the inequalities in \eqref{lbbound} follow from (\textbf{A1}) and (\textbf{A3}). \colr{The bound $\EE\norm{\g_{t-\delta_i(t)}^i}\leq V_i$ follows from \textbf{A3} and Jensen's inequality which implies that \small$\EE\norm{\g_{t-\delta_i(t)}^i}\leq\sqrt{\EE\norm{\g_{t-\delta_i(t)}^i}^2}$\normalsize.} Given $1\leq i,j \leq K$ and {$t \geq {\ell} \geq 1$}, it follows that
									\begin{align}
									\EE\norm{\lam_{t}^{i}-\lam_{{\ell}}^{j}} &\leq \sum_{k=1}^{i}\EE\norm{\lam_t^{k}-\lam_{t}^{k-1}} \nonumber\\ &\hspace{-1.5cm}+\sum_{s={\ell}+1}^{t-1}\sum_{k=1}^{K}\EE\norm{\lam_s^{k}-\lam_{s}^{k-1}} +\sum_{k=j+1}^{K}\EE\norm{\lam_{{\ell}}^{k}-\lam_{{\ell}}^{k-1}}\nonumber\\
									&\hspace{-1.5cm}\leq \ep_t \sum_{k=1}^{i} V_k+ \left(\sum_{s={\ell}+1}^{t-1}\ep_s\right)\left(\sum_{k=1}^{K} V_k\right)+\ep_{{\ell}}\sum\limits_{k=j+1}^{K} V_k\nonumber
									\\
									&{\hspace{-1.5cm}\leq \ep_t (iV)+ \left(\sum_{s={\ell}+1}^{t-1}\ep_s\right)(KV)+\ep_{{\ell}}(K-j)V}\label{new_one}
									\\
									&{\hspace{-1.5cm}\leq \ep_{{\ell}} V\left[i+(t-{\ell}-1)K+K-j\right]}\label{lambda_bound}
									\\
									&{\hspace{-1.5cm}\leq \ep_{{\ell}} V\left[\abs{i-j}+K(t-{\ell})\right]}\label{last}
									\end{align}\normalsize
									{where \eqref{new_one} is obtained by substituting $V = \max_i V_i$ and \eqref{lambda_bound} follows since $\ep_{{\ell}} \geq \ep_s$ for all ${\ell} \leq s \leq t$.} The result in \eqref{last} holds from the inequality $(i-j)\leq |i-j|$.  
									Further, for $t\geq 1$ and $1\leq i\leq K$, let $\F^i_t$ be the $\sigma$-algebra generated by the random variables \vspace{0mm}
									\begin{align}\label{sigma}
									\{\h_1^1, \ldots, \h_1^K, \h_2^1 \ldots, \h_{t-1}^K, \h_t^1,\ldots,\h_t^i\}.
									\end{align}
									where $\h_t^i$ is the random state variables observed at node $i$ at time $t$. 	With this definition, it holds that 
									\begin{align}
									\Ex{\g^i_{t-\delta_{i}(t)}\mid\mathcal{F}_{t-\delta_i(t)}^{i-1}}=\colr{\partial D^i(\lam_{t-\tau_i(t)}^{i-1})} \label{gdelta}
									\end{align}
									since $\g_{t-\delta_i(t)}^i$ may depend on $\lb_s^{i-1}$ only for $s \leq {t-\tau_i(t)}$.\\
									
									\begin{IEEEproof} The proof is organized into two parts. Subsection \ref{sec-bound} develops a bound on the optimality gap in \eqref{dual_asynch}, in terms of $I_1$. Subsequently, Subsection \ref{sec-i1} develops the required bound on $I_1$. 
										
\subsubsection{Bound on the optimality gap}\label{sec-bound}
An upper bound on $\norm{\lb_t^K-\lams}^2$ is developed by making use of the form of the updates in \eqref{general} for all nodes $1\leq i \leq K$. The bound follows from the use of triangle inequality and the moment bounds in (\textbf{A3}). Further, the bounded delay assumption (\textbf{A4}) enter through the use of \eqref{last}. 

Observe from the updates in \eqref{general} that
										\begin{align}
										\norm{\lb_t^i-\lams}^2 &=  \norm{\px{\lam_t^{i-1} -\ep_t\g_{t-\delta_i(t)}^i}-\lams}^2 \nonumber\\
										&\leq \norm{\lam_t^{i-1} -\ep_t\g_{t-\delta_i(t)}^i-\lams}^2 \label{non_exp}\\
										& \hspace{-1.5cm}= \norm{\lam_t^{i-1}-\lams}^2 -2\ep_t \ip{\g_{t-\delta_i(t)}^i, \lam_t^{i-1}-\lams} +\ep_t^2 \norm{\g_{t-\delta_i(t)}^i}^2\nonumber\\
										& \hspace{-1.5cm}= \norm{\lam_t^{i-1}-\lams}^2 -2\ep_t \ip{\g_{t-\delta_i(t)}^i, \lam_{t-\tau_i(t)}^{i-1}-\lams} \nonumber\\
										&\hspace{-1cm}+\ep_t^2 \norm{\g_{t-\delta_i(t)}^i}^2-2\ep_t \ip{\g_{t-\delta_i(t)}^i, \lam_t^{i-1}-\lam_{t-\tau_i(t)}^{i-1}}\label{dual1}
										\end{align}
										where \eqref{non_exp} follows form (\textbf{A1}), and the term $2\ep_t\ip{\g^i_{t-\delta_i(t)}, \lb^{i-1}_{t-\tau_i(t)}}$ has been added and subtracted to obtain \eqref{dual1}. Taking expectations on both sides and summing over all $1\leq i \leq K$ and $1\leq t\leq T$, we obtain
										\begin{align}
										\EE{\norm{{\lb_T^K}-\lams}^2} &\!\!\!= \EE{\norm{\lb_1^{0}-\lams}^2} + {\sum_{t=1}^{T}\sum_{i=1}^{K}} \ep_t^2\EE{\norm{\g_{t-\delta_i(t)}^i}^2}\! \!+ \!\!I_1 \nonumber\\
										&\!\!\!-2 {\sum_{t=1}^{T}\sum_{i=1}^{K}}\ep_t \Ex{\ip{\g_{t-\delta_i(t)}^i, \lam_{t-\tau_i(t)}^{i-1}\!\!-\!\!\lams}} \label{dual2}
										\end{align}\normalsize
										where $I_1$ is as defined in Lemma \ref{lemma2}. Deferring the bound on $I_1$ to Subsection \ref{sec-i1}, the last term in \eqref{dual2} is analyzed first. In particular, it holds from \eqref{gdelta} that
										\begin{align}
										&\Ex{\ip{\g_{t-\delta_i(t)}^i, \lam_{t-\tau_i(t)}^{i-1}-\lams}} \nonumber\\
										&\hspace{1cm}= \Ex{\ip{\Ex{\g_{t-\delta_i(t)}^i \mid \mathcal{F}_{t-\delta_i(t)}^{i-1}},\lam_{t-\tau_i(t)}^{i-1}-\lams}} \nonumber\\
										&\hspace{1cm}= \Ex{\ip{\colr{\partial D^i(\lam_{t-\tau_i(t)}^{i-1})},\lam_{t-\tau_i(t)}^{i-1}-\lams}}.\label{63}
										\end{align}
										Further, since the functions $D^i(\lb)$ are convex, it holds that
										\begin{align}
										&-\ip{\colr{\partial} D^i(\lam_{t-\tau_i(t)}^{i-1}),\lam_{t-\tau_i(t)}^{i-1}-\lams} \leq D^i(\lams) - D^i(\lam_{t-\tau_i(t)}^{i-1}) \nonumber\\
										&= D^i(\lams) - D^i({\lb_{t}^0}) + D^i({\lb_{t}^0}) - D^i(\lam_{t-\tau_i(t)}^{i-1}) \label{eq0}\\
										&\leq D^i(\lams) - D^i({\lb_{t}^0}) + \ip{\colr{\partial} D^i({\lb_{t}^0}), {\lb_{t}^0} -\lam_{t-\tau_i(t)}^{i-1}} \label{eq1}\\
										&\leq D^i(\lams) - D^i({\lb_{t}^0}) + V_i\norm{{{\lb_{t}^0}}-\lam_{t-\tau_i(t)}^{i-1}}\label{eq2} 
										\end{align}
										where \eqref{eq0}-\eqref{eq1} follow from the first order convexity condition for $D^i$ and \eqref{eq2} follows from the use of triangle inequality, and the fact that given any $\lb \in \L$, 
										\begin{align}
										\norm{\colr{\partial D^i(\lb)}} = \norm{\EE g_t^i(\lb)} \leq \sqrt{\EE\norm{g_t^i(\lb)}^2} \leq V_i. \label{vibound} 
										\end{align}
										For the last term in \eqref{eq2}, taking expectation and utilizing the result in \eqref{last}, it follows that
										\begin{align}
										\EE\norm{{{\lb_{t}^0}}-\lam_{t-\tau_i(t)}^{i-1}} &\leq {\ep_{t-\tau_i(t)} \left[(i-1)V+KV(\tau_i(t))\right]}\label{lbexp2}
										\end{align}
										for $1\leq i\leq K$. Finally, substituting \eqref{eq2}, \eqref{lbexp2} in \eqref{dual2}, and using (\textbf{A3}), it follows that
										\begin{align}
										\EE\norm{{\lb_T^K} - \lams}^2 \leq& \norm{\lb_1^0-\lams}^2 +{\sum_{t=1}^{T}\sum_{i=1}^{K}}\ep_t^2V_i^2 + I_1\nonumber\\&- 2{\sum_{t=1}^{T}\sum_{i=1}^{K}} \ep_t\left[\EE D^i({\lb_{t}^0}) - {D^i(\lams)}\right]
										\nonumber\\&+2{\sum_{t=1}^{T}\sum_{i=1}^{K}}\ep_t{\ep_{t-\tau_i(t)} \!V_i\left[(\!i\!-\!1\!)V+KV(\tau_i(t))\right]} \label{lbexp3}.
										\end{align}
										Since the left-hand side is non-negative and $\norm{\lb_1^0 - \lams} \leq B_0$, the first part of Lemma \ref{lemma2} is obtained simply by rearranging the terms in \eqref{lbexp3} 
										\begin{align}
										2{\sum_{t=1}^{T}\sum_{i=1}^{K}} \ep_t\left[\EE D^i({\lb_{t}^0}) - {D^i(\lams)}\right] &\leq B_0^2 +  I_1 \nonumber\\
										&\hspace{-4.5cm}+ {\sum_{t=1}^{T}\sum_{i=1}^{K}}\left(\ep_t^2 V_i^2 + 2\ep_t{\ep_{t-\tau_i(t)} V_i\left[(i-1)V+\tau KV\right]}\right)\nonumber
										\\
										&\leq \!\!B_0^2\!\! +\!\! I_0\!
										\! + I_1  \label{bet}
										\end{align} 
										where the first inequality in \eqref{bet} follows since $\tau_i(t)\leq \tau$, and $\ep_t$ is non-increasing sequence. Finally, the second inequality in \eqref{bet} follows from substituting $V_i \leq V$ for all $1\leq i \leq K$, and $I_0$ is as defined in Lemma \ref{lemma2}.

\subsubsection{Bound on $I_1$}\label{sec-i1}
In order to derive a bound on $I_1$, we make use of the Cauchy-Schwartz inequality as follows:
\begin{align}
										\Ex{\ip{\g_{t-\delta_i(t)}^i, \lam_{t-\tau_i(t)}^{i-1}-\lam_t^{i-1}}}&\leq V_i\EE\norm{\lb_{t}^{i-1} -\lb_{t-\tau_i(t)}^{i-1}} \nonumber\\
										& {\hspace{-2cm}\leq \tau_i(t)\ep_{t-\tau_i(t)} V_iKV} \label{lipbound1}
										\end{align}
										where \eqref{lipbound1} follows from {\eqref{last}}. Consequently, 
										\begin{align}
										I_1 \!\! &\leq \!\!{\sum_{t=1}^{T}\!\!\sum_{i=1}^{K}}2\ep_t {\ep_{t-\tau_i(t)} \tau_i(t)V_iKV}\leq2\tau KV{\sum_{t=1}^{T}\sum_{i=1}^{K}} V_i\ep_t\ep_{t-\tau_i(t)}\label{i1bound}
										\end{align} 
										where \eqref{i1bound} utilizes the bounds $\tau_i(t)\leq \tau$. Finally, substituting $V = \max_i V_i$, we obtain ${I_1 \leq 2\tau KV^2{\sum_{t=1}^{T}\sum_{i=1}^{K}} \ep_t \ep_{t-\tau_i(t)}}$										
										which is the required bound. 
									\end{IEEEproof}	
									\vspace{-5mm}
\section{Proof of Lemma~\ref{dual_B}}\label{boundedness}
Here, we establish that the dual iterates always stay bounded, thanks to the Slater's condition in (\textbf{A6}). The proof begins with establishing an upper bound on the per-iteration increase in the value of $\EE\norm{{\colr{\lb_t^0}} - \lams}^2$, and subsequently utilizes an induction argument to derive the following bound for all $t \geq 1$:
\vspace{-5mm}
\begin{align}\label{main_bound}
\EE\norm{\colr{\lb_{t}}} &\leq 2\norm{\lams}\!\!+\!\!\max \bigg\{\norm{\lb_1},\frac{\theta}{C}\left[\mathsf{D}-\sk \Ex{f^i(\tilde{\x}^i)}\right]\nonumber\\
&\hspace{2.3cm}+ \frac{\ep\theta KV^2}{2C}+ \frac{2\theta\ep \tau \bar{V}}{C}{+\ep VK}\bigg\}
\end{align}\normalsize  
where, $\theta$ and $C$ are positive constants, $\tilde{V}\!\! \!=\!\!\! V^2K(K\!\!-\!\!1)$, $\bar{V} \!\!\!=\!\!\! 2K^2V^2$, and $\{\bar{\x}^i\}$ is a slater point of \eqref{P1}. Since $\norm{\lam^\star}$ is bounded, the right hand side of  \eqref{main_bound} serves as the bound on $\EE\norm{\lam_t}$ .

									\begin{IEEEproof}
									 \colr{In order to prove \eqref{main_bound}, we will instead establish a more general result that takes the form:
	\begin{align}\label{bounded}
	\EE\norm{\lb_t^0 - \lams} \leq &\max \bigg\{\norm{\lb_1^0-\lams},\frac{\theta}{C}\left[\mathsf{D}-\sk \Ex{f^i(\tilde{\x}^i)}\right]\nonumber\\
	&\hspace{0cm}+\! \frac{\ep\theta KV^2}{2C}\!+\! \frac{2\theta\ep \tau \bar{V}}{C}+\norm{\lams}{+\ep VK}\bigg\}.
	\end{align}
	where, recall that $\lam_t^0 = \lam_t$ and $\lam_1^0 = \lam_1$. The desired result in \eqref{main_bound} will follow by applying the triangle inequality to \eqref{bounded}. The proof of \eqref{bounded} follows via induction. It can be seen that the inequality in \eqref{bounded} holds trivially for the base case of $t=1$. As part of the inductive hypothesis, assume that \eqref{bounded} holds for $t$ where $t\geq1$. It remains to show that it also holds for $t+1$.} We split the argument into the following two cases. 
										
										\emph{\textbf{Case 1. } $\EE D({\lb_{t}^0}) > \mathsf{D}+\ep {{\tilde{V}}/2}+ {\ep \tau \bar{V}}\normalsize$:} In this case, it holds that $\EE\norm{\colr{\lb_{t+1}^0} - \lams}^2 \leq \EE\norm{{\lb_{t}^0}-\lams}^2\normalsize$. \colr{Consequently, the induction hypothesis for time $t$ implies that \eqref{bounded} also holds for time $t+1$}.
										
										\emph{\textbf{Case 2. }  $\EE D({\lb_{t}^0}) \leq { \mathsf{D}+\ep {{\tilde{V}}/2}+ {\ep \tau \bar{V}}}$:} Recall that the dual function in \eqref{dual_exp_1} is defined as
										\begin{align}\label{appen1}
										D({\lb_{t}^0})&=\max_{\x^i\in\mathcal{X},\ \p_t^i\in\mathcal{P}_t}  \sk \left[f^i(\x^i)+\ip{{\lb_{t}^0},\EE\left[ \g_t^i(\p_t^i,\x^i)\right]}\right] \nonumber\\
										&\geq \sk \left[f^i(\tilde{\x}^i)+\ip{{\lb_{t}^0},\EE\left[ \g_t^i(\tilde{\p}_t^i,\tilde{\x}^i)\right]}\right]
										\end{align}\normalsize
										where $\{\tilde{\x}^i,\{\tilde{\p}_t^i\}_{t\geq 1}\}_{i=1}^K$ is a strictly feasible (Slater) solution to \eqref{P1}. From (\textbf{A6}), such a strictly feasible solution exists and satisfies $\EE\left[ \g_t^i(\tilde{\p}_t^i,\tilde{\x}^i)\right] > C > 0$. Substituting into \eqref{appen1}, and rearranging, we obtain\vspace{0mm}
										\begin{align}\label{bound_on_dual}
										\ip{\boldsymbol{1},{\lb_{t}^0}}&\leq\frac{1}{C}\left[D({\lb_{t}^0})-\sk f^i(\tilde{\x}^i)\right] 
										\end{align}
										Since ${\lb_{t}^0} \succeq 0$, it follows from equivalence of norms $\norm{{\lb_{t}^0}} \leq \theta\norm{{\lb_{t}^0}}_1 = \theta\ip{\boldsymbol{1},{\lb_{t}^0}}$. Therefore, taking expectations in \eqref{bound_on_dual} yields
										\begin{align}
										\EE\norm{{\lb_{t}^0}} &\leq \frac{\theta}{C}\left[\Ex{D({\lb_{t}^0})}-\sk \EE f^i({\tilde{\x}^i})\right] \\
										&\leq \frac{\theta}{C}\left[{ \mathsf{D}+\ep {{\tilde{V}}/2}+ {\ep \tau \bar{V}}}-\sk \EE f^i({\tilde{\x}^i})\right]\label{case2bound}
										\end{align}\normalsize
										where the assumption for Case 2 has been used in \eqref{case2bound}. Finally, the use of triangle inequality and the bound in {\eqref{last}} yields
										\begin{align}
										\EE\norm{\colr{\lb_{t+1}^0}-\lams} &\leq \EE\norm{{\lb_{t}^0}} + \EE\norm{\colr{\lb_{t+1}^0}-{\lb_{t}^0}} + \norm{\lams}\\
										&\leq \EE\norm{{\lb_{t}^0}} + \ep VK + \norm{\lams}
										\end{align}
										which, together with \eqref{case2bound}, yields \eqref{bounded} \colr{for $t+1$}. \colr{Therefore by mathematical induction, the inequality in \eqref{bounded} holds for all $t\geq 1$.} Finally, using \eqref{bounded} and triangle inequality, we obtain the result in \eqref{main_bound} since $\colr{\lb_{t}}=\colr{\lb_{t}^0}$.
\end{IEEEproof} \vspace{-4mm}									
									\section{Proof of lemma~\ref{lemma3}} \label{lemma3proof}
									\begin{IEEEproof}
										The proof establishes a lower bound for the running average of the primal objective function, calculated at the primal iterates. The lower bound depends upon the dual optimal value, dual initialization, and the maximum delay bound $\tau$. For ease of  exposition, the proof begins with re-arranging the optimality gap in the form required by Lemma \ref{lemma3} and subsequently analyzing the resulting terms. The full proof is split into various parts that develop separate bounds on the terms $I_2$, $I_3$, and $I_4$. Since (\textbf{A7}) is required to establish Lemma \ref{lemma3}, the dual function $D$ has to be differentiable. 
										
Since the functions $f^i$ are concave, the expected value of the primal objective can be written as\vspace{0mm}
										\begin{align}
										&\Ex{\sk f^i\left(\colr{\bar{\x}^i_T}\right)}\geq \frac{1}{T}\Ex{{\sum_{t=1}^{T}\sum_{i=1}^{K}} f^i(\x^i_{t})}\label{primal1}\\
										&\!=\!\!\frac{1}{T}\!{\sum_{t=1}^{T}\!\sum_{i=1}^{K}}\EE\Big[\!f^i(\x^i_{t})\!\!+\!\!\ip{\lamb\!,\!\nabla D^i(\lamb\!)\!}\!-\!\ip{\lamb\!,\!\!\nabla D^i(\lamb)\!}\Big] \nonumber\\
										&= \!\!\frac{1}{T}{\!\sum_{t=1}^{T}\!\sum_{i=1}^{K}} \EE D^i(\lamb) \!-\! \frac{1}{T}\!{\sum_{t=1}^{T}\!\sum_{i=1}^{K}} \EE\ip{\lamb\!\!,\nabla D^i(\lamb)}\label{lemma_3}
										\end{align}
										\colr{Consider the following expression where we simply add subtract  $D^i({\lb_{t}})$ as follows 					\begin{align}
											\sk D^i(\lamb)&=\sk [D^i(\lamb)+ D^i({\lb_{t}})-D^i({\lb_{t}})]\nonumber\\
											&=\sk D^i({\lb_{t}})+\sk(D^i(\lamb)-D^i({\lb_{t}})) \nonumber\\
											&\geq \mathsf{D} + \sk(D^i(\lamb)-D^i({\lb_{t}})) \label{intrp_am_tau}
											\end{align}}
										\colr{	where \eqref{intrp_am_tau} follows since $\mathsf{D} = \sk D^i(\lams) \leq \sk D^i(\lb)$ for all $\lb \in \L$. Taking the expectation on both sides of \eqref{intrp_am_tau} and substituting the result into \eqref{lemma_3}, we obtain }
										\colr{\begin{align}
										 \Ex{\sk f^i\left(\colr{\bar{\x}^i_T}\right)}\geq
										 & \mathsf{D} - I_2 - I_3\label{primal_end}
										\end{align}}
										where $I_2$ and $I_3$ are as defined in Lemma \ref{lemma3}. The rest of the proof proceeds simply by developing bounds on $I_2$ and $I_3$. 
										\setcounter{subsubsection}{0}
\subsubsection{Bound on $I_2$} The bound on $I_2$ follows simply from the moment bounds in (\textbf{A3}) and the Cauchy-Schwartz inequality. We begin with the following observation
					\normalsize 
									 Since the functions $D^i$ are convex, it holds that
										\begin{align}
										\EE[D^i({\lb_{t}})-D^i(\lamb)] &\leq  \EE\ip{\nabla D^i({\lb_{t}}), {\lb_{t}}-\lamb}\label{incremental_ineq} \\
										&\hspace{-1cm}\leq \EE\norm{\nabla D^i({\lb_{t}})}\EE\norm{{\lb_{t}}-\lamb}\label{i3bound1} \\
										&\leq V_i\ep{(i-1)V} \label{i3bound2}
										\end{align}
										where \eqref{i3bound1} uses the Cauchy-Schwartz inequality, while \eqref{i3bound2} uses \eqref{last} and \eqref{vibound}. Therefore, substituting \eqref{i3bound2} into the expression for $I_2$ and rearranging, we obtain
										$I_2 \leq{ \ep {V}\sum_{i=1}^K(i-1)V_i}.$
										Finally, the required bound in Lemma \ref{lemma3} is obtained by substituting $V = \max_i V_i$. 

	\subsubsection{Bound on $I_3$}
The bound in $I_3$ follows from setting aside the error due to asynchrony $I_4$, and developing a bound on the remaining terms by telescopically summing the bounds on $\norm{\lam_t^i}$ over all $1\leq i\leq K$ and $1\leq t\leq T$.

Since $\mathbf{0}\in\Lambda$ is a feasible dual solution, using the form of the updates in \eqref{dual_sub} and expanding as in \eqref{non_exp}, it follows that
										\begin{align}
										\!\!\!\!\!\norm{\lambm}^2&\leq  \norm{\lamb}^2 +\norm{\ep\g_{t-\delta_i(t)}^i}^2-2\ep\ip{\lamb, \g_{t-\delta_i(t)}^i}.
										\end{align}
										Adding the term $2\ep\ip{\lb_t^{i-1},\nabla D^i(\lb_t^{i-1})}$ on both sides, and rearranging, we obtain
										\begin{align}
										2\ep\ip{\lamb,\nabla D^i(\lamb)}\leq& \norm{\lamb}^2-\norm{\lambm}^2+\norm{\ep\g_{t-\delta_i(t)}^i}^2\nonumber\\ &-2\ep\ip{\lamb,\e_{t,\delta_i(t)}^{i}}
										\end{align}
										where $\e_{t,\delta_i(t)}^i$ is as defined in \eqref{error_asyn2}. Summing over $i=1,\ldots,K$ and  $t=1,\cdots,T$, taking expectation, and utilizing (\textbf{A3}), it follows that
										\begin{align}
										I_3 &\leq \frac{\norm{\lb_1}^2}{2\ep T} + \frac{\ep }{2}\sk V_i^2 + I_4 \leq \frac{\norm{\lb_1}^2}{2\ep T} + \frac{\ep KV^2 }{2} + I_4 \label{i4bound2}
										\end{align}
										where $I_4$ is as defined in Lemma \ref{lemma3} and the \eqref{i4bound2} uses $V = \max_iV_i$. 

\subsubsection{Bound on $I_4$}
The term $I_4$ collects the error from the terms that arise due to asynchrony. A bound on $I_4$ is developed from the use of the delay bound assumption in (\textbf{A4}).	Adding and subtracting $\ip{\lb_{t}^{i-1}, \nabla D^i(\lb_{t-\tau_i(t)}^{i-1})}$ to each summand of $I_4$, we obtain
										\begin{align}\label{inbet}
										I_4=&\frac{1}{T}{\sum_{t=1}^{T}\sum_{i=1}^{K}}\Ex{\ip{\lamb,\nabla D^i(\lamb)-\nabla D^i(\lb_{t-\tau_i(t)}^{i-1})}}\nonumber\\
										&+{\sum_{t=1}^{T}\sum_{i=1}^{K}}\Ex{\ip{\lamb, \nabla D^i(\lb_{t-\tau_i(t)}^{i-1})-\g_{t-\delta_i(t)}^i}}.
										\end{align} 
										Of these, the first term {in \eqref{inbet}} can be bounded using the bound in Lemma \col{\ref{dual_B}} and the Cauchy-Schwartz inequality, by observing that\vspace{-3mm}
										\begin{align}
										&\Ex{\ip{\lamb,\nabla D^i(\lamb)-\nabla D^i(\lb_{t-\tau_i(t)}^{i-1})}} \nonumber\\
										&\leq B\EE\norm{\nabla D^i(\!\lamb\!)\!-\!\nabla D^i(\lb_{t-\tau_i(t)}^{i-1})}\! \nonumber\\&\leq \!B L^i\EE\norm{\lamb\!-\!\lb_{t-\tau_i(t)}^{i-1}} \label{primallip}\\
										&\leq B L^i \ep {\tau KV}\label{primallbbound}
										\end{align}
										where \eqref{primallip} follows from (\textbf{A7}) and \eqref{primallbbound} from the bound developed in \eqref{last}.
										
										For the second term {in \eqref{inbet}}, recalling the definition of $\F_{t-\delta_i(t)}^{i-1}$ from Appendix \ref{lemma2_proof}, observe that although $\Ex{\lb_{t}^i \mid \F_{t-\delta_i(t)}^{i-1}} \neq \lb_t^i$, there exists some $\kappa_i(t) \leq t$ such that  
										\begin{align}
										\Ex{\lb_{\kappa_i(t)}^{i-1} \mid \F_{t-\delta_i(t)}^{i-1}} = \lb_{\kappa_i(t)}^{i-1}\label{kappa}.
										\end{align}
										Indeed, observe that {$\kappa_i(t) \geq t-\delta_i(t)$ since $\lb_{t-\tau_i(t)}^{i-1}$} only depends on random variables contained in $\F_{t-\delta_i(t)}^{i-1}$. \colr{The subsequent bounds hold for any $\kappa_i(t)$ that satisfies \eqref{kappa}, including for the worst case when $\kappa_i(t) = t-\delta_i(t)$}. It follows that 										\begin{align}
										&\Ex{\ip{\nabla D^i(\lam_{t-{\tau_i(t)}}^{i-1})-\g_{t-\delta_i(t)}^i,\lb_{t}^{i-1}}} \nonumber
										\\
										& = \Ex{\Ex{\ip{\nabla D^i(\lam_{t-{\tau_i(t)}}^{i-1})-\g_{t-\delta_i(t)}^i,\lb_{t}^{i-1}}\mid\mathcal{F}_{t-\delta_i(t)}^{i-1}}}\nonumber
										\\
										& = \EE\bigg[\EE\bigg[\ip{\nabla D^i(\lam_{t-{\tau_i(t)}}^{i-1})-\g_{t-\delta_i(t)}^i,\lb_{\kappa_i(t)}^{i-1}}\mid\mathcal{F}_{t-\delta_i(t)}^{i-1}\bigg]\bigg] \nonumber
										\\
										&\ \ \ + \Ex{\ip{\nabla D^i(\lam_{t-{\tau_i(t)}}^{i-1})-\g_{t-\delta_i(t)}^i,\lb_t^{i-1} - \lb_{\kappa_i(t)}^{i-1}}} \label{i2term}.
										\end{align}
										From \eqref{gdelta} and \eqref{kappa}, it follows that the first summand in \eqref{i2term} is zero. The second summand can be bounded by using the Cauchy-Schwartz inequality and the bounds in (\textbf{A4}) and \eqref{last} as follows:
										\begin{align}
										&\Ex{\ip{\nabla D^i(\lam_{t-{\tau_i(t)}}^{i-1})-\g_{t-\delta_i(t)}^i,\lb_t^{i-1} - \lb_{\kappa_i(t)}^{i-1}}} \nonumber\\
										&\leq \Ex{\norm{\nabla D^i(\lam_{t-{\tau_i(t)}}^{i-1})-\g_{t-\delta_i(t)}^i}}\Ex{\norm{\lb_t^{i-1} - \lb_{\kappa_i(t)}^{i-1}}} \nonumber\\
										&\leq \ep {2V_i} {(t-\kappa_i(t)) {K}V} \label{i2bound1}\\
										& \leq \ep{2V_i} {\tau {K}V} \label{i2bound}
										\end{align}
										where the inequality in \eqref{i2bound} follows since $t-\tau_i(t) \leq t-\delta_i(t) \leq \kappa_i(t) \leq t$. Finally, substituting \eqref{i2bound} and \eqref{primallbbound} into \eqref{inbet} yields
										\begin{align}
										\!\!\!\!I_4 &\leq \sum_{i=1}^K { \ep\tau KV}(B L^i+{2V_i})\leq {\ep\tau K^2V(BL+2V)}
										\end{align}
										which together with \eqref{i4bound2} gives the desired bound. 
									\end{IEEEproof} 
									\vspace{0mm}									

							\vspace{0mm}
										
										\footnotesize
										\bibliographystyle{IEEEtran} 
										\bibliography{IEEEabrv,references}
										
									\end{document}